\documentclass{article}
\usepackage{amsfonts,amsthm,amssymb,amsmath} \textwidth=160truemm \textheight=240truemm \evensidemargin=0mm \oddsidemargin=0mm \topmargin=0mm \headsep=0mm\parindent=2em \allowdisplaybreaks

\usepackage[noadjust]{cite}
\usepackage{shuffle,yfonts}
\usepackage[T1]{fontenc}
\usepackage{lmodern}
\DeclareFontFamily{U}{shuffle}{}
\DeclareFontShape{U}{shuffle}{m}{n}{ <-8>shuffle7 <8->shuffle10}{}

\DeclareMathOperator\sh{{{sh}}}
\DeclareMathOperator\ch{{{ch}}}
\DeclareMathOperator\rre{{{Re}}}
\let\th\relax
\DeclareMathOperator\th{{{th}}}
\DeclareMathOperator\cth{{{cth}}}
\DeclareMathOperator\sech{{{sech}}}
\DeclareMathOperator\csch{{{csch}}}

\newcommand\nc{\newcommand}

\newcommand{\smbb}[1]{\left(#1\right)}
\newcommand{\uud}{\,\mathrm{d}}
\newcommand{\uue}{\,\mathrm{e}}

\nc{\ES}{\mathsf {ES}}
\nc{\MZV}{\mathsf {MZV}}
\nc{\MtV}{\mathsf {MtV}}
\nc{\MTV}{\mathsf {MTV}}
\nc{\MSV}{\mathsf {MSV}}
\nc{\MMV}{\mathsf {MMV}}
\nc{\MMVo}{\mathsf {MMVo}}
\nc{\MMVe}{\mathsf {MMVe}}
\nc{\AMMV}{\mathsf {AMMV}}
\nc{\sha}{\shuffle}

 \nc{\tlg}{{\tilde{g}}}
 \nc{\tlp}{{\tilde{p}}}
\nc{\calC}{{\mathcal C}}
\nc{\calR}{{\mathcal R}}
\nc{\calS}{{\mathcal S}}
\nc{\gd}{\delta}
\nc{\gs}{\sigma}
\nc{\ola}{\overleftarrow}
\nc{\ora}{\overrightarrow}
\nc{\lra}{\longrightarrow}
\nc{\Lra}{\Longrightarrow}
\nc\Res{{\rm Res}}
\nc\setX{{\mathsf{X}}}
\nc\fA{{\mathfrak{A}}}
\nc\evaM{{\texttt{M}}}
\nc\evaML{{\text{\em{\texttt{M}}}}}
\nc\z{{\texttt{z}}}
\nc\txp{{\tx_1}} % textstyle x positive 1
\nc\txn{{\tx_{-1}}} % textstyle x negative 1
\nc\neo{{1}}
\nc{\yi}{{1}}
\nc\one{{-1}}
\nc\gf{{\phi}}
\nc\ga{{\alpha}}
\nc\gb{{\beta}}
\nc\gt{{\tau}}
\nc\gD{{\Delta}}
\nc\gG{{\Gamma}}
\nc\om{{\omega}}
\nc\omn{\omega_{-1}}
\nc\omz{\omega_0}
\nc\omp{\omega_{1}}
\nc\eps{{\varepsilon}}
\nc\ta{{\texttt{a}}}
\nc\ty{{\texttt{y}}}
\nc\tx{{\texttt{x}}}
\nc\td{{\texttt{d}}}
\nc\tz{{\texttt{z}}}
\nc{\bfp}{{\bf p}}
\nc{\bfq}{{\bf q}}
\nc{\bfr}{{\bf r}}
\nc{\bfu}{{\bf u}}
\nc{\bfv}{{\bf v}}
\nc{\bfw}{{\bf w}}
\nc{\bfy}{{\bf y}}
\nc{\bfe}{{\textbf{\em{e}}}}
\nc{\bfi}{{\textbf{\em{i}}}}
\nc{\bfj}{{\textbf{\em{j}}}}
\nc{\bfk}{{\textbf{\em{k}}}}
\nc{\bfkk}{{\text{\textbf{k}}}}
\nc{\bfl}{{\textbf{\em{l}}}}
\nc{\bfm}{{\textbf{\em{m}}}}
\nc{\bfn}{{\textbf{\em{n}}}}
\nc{\bfnn}{{\text{\textbf{n}}}}
\nc{\bfs}{{\textbf{\em{s}}}}
\nc{\bft}{{\textbf{\em{t}}}}
\nc{\bfx}{{\textbf{\em{x}}}}
\nc{\bfz}{{\textbf{\em{z}}}}
\nc\bfgl{{\boldsymbol \lambda}}
\nc\bfsi{{\boldsymbol \sigma}}
\nc\bfet{{\boldsymbol \eta}}
\nc\bfeta{{\boldsymbol \eta}}
\nc\bfeps{{\boldsymbol \varepsilon}}
\nc\bfone{{\bf 1}}
\nc{\baru}{{\bar u}}
\nc{\bmu}{{\bar\mu}}
\nc{\myone}{{1}}
\nc{\reg}{{\rm R}}

\allowdisplaybreaks

\let\oldsection\section
\renewcommand\section{\setcounter{equation}{0}\oldsection}

\allowdisplaybreaks

\DeclareMathOperator*{\dep}{dep}
\DeclareMathOperator{\Li}{Li}

\DeclareMathOperator{\cl}{Cl}

\nc{\binn}{{\binom{2n}{n}}}
\nc{\rint}[4]{\int_{#1}^{#2}\hskip#3{\rm R}\hskip#4}
\nc\UU{\mbox{\bfseries U}}
\nc\FF{\mbox{\bfseries \itshape F}}
\nc\h{\mbox{\bfseries \itshape h}}\nc\dd{\mbox{d}}
\nc\g{\mbox{\bfseries \itshape g}}
\nc\xx{\mbox{\bfseries \itshape x}}
\def\R{\mathbb{R}}

\def\vU{\mathsf{U}}
\def\CMZV{\mathsf{CMZV}}
\def\N{\mathbb{N}}
\def\Z{\mathbb{Z}}
\def\Q{\mathbb{Q}}
\def\CC{\mathbb{C}}

\def\ze{\zeta}

\def\xx{\left(\frac{1-x}{1+x} \right)}

\nc\divg{{\text{div}}}
\theoremstyle{plain}
\newtheorem{thm}{Theorem}[section]
\newtheorem{lem}[thm]{Lemma}
\newtheorem{cor}[thm]{Corollary}
\newtheorem{conj}[thm]{Conjecture}
\newtheorem{prop}[thm]{Proposition}
\theoremstyle{definition}
\newtheorem{defn}{Definition}[section]
\newtheorem{re}[thm]{Remark}
\newtheorem{ex}[thm]{Example}

\begin{document}
 \title{\bf Ap\'ery-Like Sums and Colored \\ Multiple Zeta Values}
\markright{Ap\'ery-Like Sums and Colored MZVs}
\author{
{Ce Xu${}^{a,}$\thanks{Email: cexu2020@ahnu.edu.cn,  ORCID 0000-0002-0059-7420. Ce Xu is supported by the General Program of Natural Science Foundation of Anhui Province (Grant No. 2508085MA014).}\ \ and Jianqiang Zhao${}^{b,}$\thanks{Email: zhaoj@ihes.fr, corresponding author, ORCID 0000-0003-1407-4230. Jianqiang Zhao is supported by the Jacobs Prize from The Bishop's School.} }\\[1mm]
\small a. School of Mathematics and Statistics, Anhui Normal University, Wuhu 241002, PRC\\
\small b. Department of Mathematics, The Bishop's School, La Jolla, CA 92037, USA}

\date{\emph{Dedicated to Professor Zhi-Wei Sun on the occasion of his 60th birthday}
}
\maketitle

\begin{abstract}
In this article, we survey recent progress in the study of Ap\'ery-like sums, which are multivariable generalizations of the two sums Ap\'ery used in his famous proof of the irrationality of $\zeta(2)$ and $\zeta(3)$. We focus exclusively on infinite sums where central binomial coefficients appear in either the numerator or the denominator. Special values of both types are closely related to colored multiple zeta values and play important roles in calculating the $\epsilon$-expansion of multiloop Feynman diagrams. Finally, we summarize several distinct approaches to computing these sums.
\end{abstract}
\footnotetext{ The authors would like to express their sincere gratitude to the anonymous reviewer for the valuable comments and suggestions, which have greatly contributed to improving the quality of this paper.}


%11M32 Multiple Dirichlet series and zeta functions and multizeta values
%11B65 Binomial coefficients; factorials; q-identities
%11B37 Recurrences
%44A05 General integral transforms
%11M06 zeta(s) and L(s,\chi)
%33B30 Higher logarithm functions

\section{Multiple zeta values}
\subsection{Historical background}
The study of the Riemann zeta function $\zeta(z)$ and its special values has long been a primary driving force in modern number theory. More than a century before Riemann defined his namesake function for a complex variable $z$, a 28-year-old Euler achieved widespread fame by solving the renowned Basel problem and finding the exact value of $\zeta(2)$. Euler provided three distinct proofs that $\zeta(2)=\pi^2/6$ in \cite{Euler1735Basel}, presenting his findings in December 1735 at the St. Petersburg Academy of Sciences. Seven years later, a letter from Goldbach in Moscow introduced Euler to several curious double sums, prompting him to begin his investigation of double zeta values -- a subject he would revisit multiple times throughout his legendary mathematical career.

In modern notation, the multiple zeta function of complex variables $z_1,\dotsc,z_d$ is defined by
\begin{equation}\label{equ:MZVdefn}
\zeta(z_1,\dotsc,z_d):=\sum\limits_{n_1>\cdots>n_d>0 } \frac{1}{n_1^{z_1}\cdots n_d^{z_d}}.
\end{equation}
It converges if $\rre(z_1+\dotsm+z_j)>j$ for all $j=1,\dotsc,d$ and similarly to $\zeta(z)$, and
it can be analytically continued to a meromorphic function to $\CC^d$ with points of indeterminacy explicitly described
(see \cite{AkiyamaEgTa2001,Zhao1999}). Its special value $\zeta(\bfk)$ at positive integer arguments $\bfk=(k_1,\dotsc,k_d)$
is called a \emph{multiple zeta value} (MZV). An MZV converges if and only if $k_1\ge 2$ in which case $\bfk$ is called \emph{admissible}.
Conventionally, $d$ and $|\mathbf{k}|=k_1+\dots+k_d$ are called the \emph{depth} and the \emph{weight} of the MZV, respectively. A celebrated result of Euler states that if the weight $w=a+b$ is odd, then the double zeta value $\zeta(a,b)$ can be expressed as a $\mathbb{Q}$-linear combination of $\zeta(w)$ and products of Riemann zeta values $\zeta(k)\zeta(w-k)$ for $2\le k\le w-2$ (see \cite[Prop.~7]{Zagier2012} for a modern constructive proof). The special case when $w=3$, often dubbed Euler's identity, is particularly enticing:
\begin{equation}\label{equ:EulerId}
\zeta(3)=\zeta(2,1).
\end{equation}

Following Euler's pioneering work, the study of MZVs remained dormant for more than two centuries until two influential papers appeared in the final decade of the 20th century. Around 1990, D. Zagier at the Max Planck Institute for Mathematics investigated the connections between MZVs, mixed Tate motives over $\mathbb{Z}$, and Vassiliev knot invariants, publishing his insights in the seminal paper \cite{Zagier1994} (see \cite{Brown2012} and \cite{Furusho2015} for more recent developments). In particular, he highlighted a crucial representation of MZVs via Chen's iterated integrals, an connection first noticed by M.~Kontsevich.

Across the Atlantic, during roughly the same time period, M.~Hoffman at the US Naval Academy began examining rational relations among MZVs \cite{Hoffman1992}, subsequently uncovering their underlying algebraic structures in his 1997 Journal of Algebra paper \cite{Hoffman1997}. The framework established in that paper has since driven a series of major advancements in the study of both MZVs and their various modern generalizations.

\subsection{Chen's theory of iterated integrals}
For $r>1$, we define inductively
\begin{equation*}
\int_a^b f_1(t)\, dt \cdots f_r(t)\, dt=
\int_a^b f_1(\gt)\, d\gt \cdots f_{r-1}(\gt) \left(\int_a^\gt f_r(t)\, dt\right) \, d\gt.
\end{equation*}
We set the integral to be 1 if $r=0$.
More generally, let $\gf_1,\dots , \gf_r$ be some $1$-forms (repetition allowed)
on a manifold $M$ and let $\ga: [0,1]\to M$ be a piecewise smooth path.
Write $\ga^* \gf_i = f_i(t) \, dt$ and define the iterated integral
\begin{equation*}
\int_\ga \gf_1 \cdots \gf_r:=
\int_0^1 f_1(t)\, dt \cdots f_r(t)\, dt.
\end{equation*}

The following results are crucial in the application of Chen's theory of the iterated path integrals. For a sketch of the proof we refer the interested reader to \cite[Lemma~2.1.2]{Zhao2016}.
\begin{lem}\label{lem:chen's}
Let $\gf_i$ $(i\ge 1)$ be $\CC$-valued $1$-forms on a manifold $M$. Then

\begin{enumerate}
\item[\upshape{(i)}] If $\ga,\gb:[0,1]\longrightarrow M$ are composable paths
(\emph{i.e.}, $\ga(1)=\gb(0)$), then
\begin{equation*}
\int_{\gb\ga} \gf_1\cdots \gf_r=\sum_{j=0}^r
\int_\gb \gf_1\cdots \gf_j\int_\ga \gf_{j+1}\cdots \gf_r,
\end{equation*}
where we set $\int_\ga \phi_1\cdots \phi_m=1$ if $m=0$ and
$\gb\ga$ denotes the composition of $\ga$ and $\gb$.

\item[\upshape{(ii)}] For every path $\ga:[0,1]\lra M$,
let $\ga^{-1}(t)=\ga(1-t)$. Then
$$ \int_{\ga^{-1}}\gf_1\cdots \gf_r = (-1)^r\int_\ga \gf_r\cdots \gf_1.$$

\item[\upshape{(iii)}] For every path $\ga$,
\end{enumerate}
\begin{equation*}
 \int_\ga \gf_1\cdots \gf_r\int_\ga \gf_{r+1}\cdots \gf_{r+s}= \int_\ga \gf_1\cdots \gf_r\sha \gf_{r+1}\cdots \gf_{r+s},
\end{equation*}
where $\sha$ is the usual shuffle product of words\footnote{The shuffle product behaves exactly as the shuffling of two decks of cards.}.
\end{lem}

For any admissible index $\bfk=(k_1,\dotsc,k_d)\in\N^d$ (i.e., $k_1\ge 2$) we have
\begin{equation}\label{equ:iteratedForm}
\zeta(\bfk)=\int_0^1 \left(\frac{dt}{t}\right)^{k_1-1} \frac{dt}{1-t}\cdots
\left(\frac{dt}{t}\right)^{k_d-1} \frac{dt}{1-t}.
\end{equation}
This can be proved easily if we replace the interval of integration by $[0,x]$ on the right-hand side,
and replace $\zeta(\bfk)$ by the \emph{single variable multiple polylogarithm}
\begin{equation}\label{equ:singleVMPL}
\Li_{\bfk}(x):=\sum\limits_{n_1>\cdots>n_d>0 } \frac{x^{n_1}}{n_1^{k_1}\cdots n_d^{k_d}}
\end{equation}
on the left-hand side, and proceed by induction on the weight using differentiation by the Fundamental Theorem of Calculus. It follows immediately by the change of variables $t_j\leftrightarrow 1-t_j$
that the following duality relation holds. For $l\in\N$, we denote by $1_l$ the sequence $(1,\dotsc,1)$ with 1
repeating $l$ times.

\begin{thm} \emph{(Duality of MZVs)} For all $k_1,\dotsc,k_d, l_1,\dotsc,l_d\in\N$ we have
\begin{equation*}
\zeta(k_1+1,1_{l_1-1},\dotsc,k_d+1,1_{l_d-1})=\zeta(l_d+1,1_{k_d-1},\dotsc,l_1+1,1_{k_1-1}).
\end{equation*}
\end{thm}

\subsection{Double shuffle relations of MZVs}
The most fundamental question concerning MZVs is the nature of their algebraic relations over the rational field $\mathbb{Q}$. Euler's evaluation $\zeta(2)=\pi^2/6$ implies that $\zeta(2)$ is transcendental, a direct consequence of the transcendence of $\pi$. His approach can be readily extended to show that every Riemann zeta value at a positive even integer, $\zeta(2n)$, is a rational multiple of $\pi^{2n}$. Conversely, while it is widely believed that $\zeta(2n+1)$ is transcendental for every $n\ge 1$, not a single case has been definitively proven so far. In 1978, Ap\'ery \cite{Apery1978} stunned the mathematical community by proving that $\zeta(3)$ is irrational. His ingenious proof relied on a meticulous manipulation of the rapidly converging series:
\begin{equation}\label{equ:AperySeries}
\zeta(3)=\frac52\sum_{n\ge1}\frac{(-1)^{n-1}}{n^3\binom{2n}{n}}.
\end{equation}

It is therefore natural to ask, more generally, whether all MZVs are transcendental. As this question is currently out of reach, we are led to the next fundamental problem: how can we find as many rational relations among MZVs as possible? The answer is conjecturally provided by the so-called regularized double shuffle relations, which, roughly speaking, arise from multiplying MZVs in two distinct ways.

When multiplying two MZVs, we can first use their defining series to express the product as a $\mathbb{Z}$-linear combination of MZVs of the same weight. For example, for all integers $m,n\ge 2$, we have
\begin{align*}
\zeta(m)\zeta(n) &= \sum_{k>l\ge 1} \frac{1}{k^ml^n} + \sum_{k=l\ge 1} \frac{1}{k^ml^n} + \sum_{l>k\ge 1} \frac{1}{k^ml^n}
= \zeta(m,n) + \zeta(m+n) + \zeta(n,m).
\end{align*}
This decomposition generalizes naturally to MZVs of arbitrary depths. In the final expression above, terms like the first and the third are called ``shuffle terms'', the second is a``stuffing term'' (i.e., obtained by merging arguments). This method of expanding a product is called the \emph{stuffle product}. The term was coined in \cite{BorweinBBL1998}, though similar algebraic structures had appeared implicitly in Cartier's work \cite{Cartier1972} in the early 1970s. As a basic example, the stuffle product yields:
\begin{equation}\label{equ:z2stuffle}
\zeta(2)^2 = 2\zeta(2,2) + \zeta(4).
\end{equation}

The second method of multiplying MZVs arises from their iterated integral expression \eqref{equ:iteratedForm}. This approach utilizes the shuffle product structure provided by Chen's Lemma~\ref{lem:chen's}(iii). For example, defining the differential forms
\begin{equation}\label{equ:set1fomr}
\om_0 = \frac{dt}{t}, \qquad \om_1 = \frac{dt}{1-t},
\end{equation}
we obtain
\begin{align}\label{equ:z2shuffle}
\zeta(2)^2 &= \int_0^1 \om_0 \om_1 \sha \om_0 \om_1 \nonumber \\
&= 2\int_0^1 \om_0 \om_1 \om_0 \om_1 + 4\int_0^1 \om_0 \om_0 \om_1 \om_1 \nonumber \\
&= 2\zeta(2,2) + 4\zeta(3,1).
\end{align}
Combining \eqref{equ:z2stuffle} and \eqref{equ:z2shuffle} yields the \emph{finite double shuffle relation} (DBSF):
\begin{equation*}
\zeta(4) = 4\zeta(3,1).
\end{equation*}
It is clear from this construction that any non-trivial finite DBSF must have a weight of at least four.

\subsection{Regularization and renormalization of MZVs}
A natural question now arises: is it possible to prove Euler's identity \eqref{equ:EulerId} by extending the above idea? This hinges on properly defining ``$\zeta(1)$'' by generalizing both the series definition \eqref{equ:MZVdefn} and the integral expression \eqref{equ:iteratedForm} to include all non-admissible MZVs. This regularization was achieved by K.~Ihara, M.~Kaneko, and D.~Zagier in their ground-breaking work \cite{IKZ2006}.

For any index vector $\mathbf{k}=(k_1,\dots,k_d)\in\mathbb{N}^d$, let $\mathbf{w}(\mathbf{k}) := \om_0^{k_1-1} \om_1 \dots \om_0^{k_d-1} \om_1$ denote its associated word. Utilizing Hoffman's algebraic framework \cite{Hoffman1997}, Ihara et al.~demonstrated that the shuffle product $\shuffle$ of words for admissible MZVs can be extended to non-admissible ones via a suitable regularization procedure. Analogously, one can define a stuffle product $\ast$ on words that mirrors the series-product structure of admissible MZVs and extend it to the non-admissible setting.

By an abuse of notation, we set $\zeta(\bfw(\bfk))=\zeta(\bfk)$ and extend $\zeta$ linearly over the space of admissible words. For any admissible indices $\bfk$ and $\bfl$,, we have
\begin{equation}\label{equ:finiteDBSF}
\zeta(\bfk)\zeta(\bfl)=\zeta\big(\bfw(\bfk) \sha \bfw(\bfl) \big)=\zeta\big(\bfw(\bfk) \ast \bfw(\bfl) \big).
\end{equation}
which is precisely the finite DBSF. The key observation in \cite{IKZ2006} is that for any $\ell\ge 0$ and
$\bfk=(1_\ell,k_{\ell+1},\dots,k_d)\in\N^d$ ($k_{\ell+1}\ge 2$ if $\ell<d$), one can determine the asymptotic expansion of the \emph{regularized iterated integral}
\begin{align*}
\zeta^{(\eps)}(\bfk):= \int_0^{1-\eps} \bfw(\bfk)= P_\ell(\log \eps)+O(\eps \log^d \eps) \quad \text{as }\eps\to 0^+,
\end{align*}
and the \emph{multiple harmonic sums}
\begin{align}\label{equ:MHS}
H_M(\bfk):= \sum\limits_{M\ge n_1>\cdots>n_d>0 } \frac{1}{n_1^{k_1}\cdots n_d^{k_d}}
=Q_\ell(\log M+\gamma)+O(M^{-1}\log^d M)  \quad \text{as } M \to \infty.
\end{align}
Here $\gamma$ is Euler's constant and $P_\ell(T), Q_\ell(T) \in \R[T]$  are polynomials of degree $\ell$ whose nonzero coefficients are all $\Q$-linear combinations of MZVs of weight $|\bfk|-j$ for the $T^j$ term. Moreover, if $\ell=0$ then $P_\ell(T)=Q_\ell(T)=\zeta(\bfk).$

\begin{defn}
The $\sha$- and $\ast$-\emph{regularization} of $\zeta(\bfk)$
are defined, respectively, by
\begin{equation}\label{equ:MZVreg}
 \zeta_\sha(\bfk):=P_\ell(T) \quad\text{and} \quad \zeta_\ast(\bfk):=Q_\ell(T).
\end{equation}
Furthermore, the value $P_\ell(0)$ is called the \emph{renormalization} or \emph{renormalized value}
of the divergent integral $\int_{0}^{1} \bfw(\bfk)$
representing the non-admissible MZV $\zeta(\bfk)$, denoted by
\begin{equation}\label{equ:renormMZV}
 \rint{0}{1}{-3.3ex}{1.5ex} \bfw(\bfk).
\end{equation}
\end{defn}

Ihara et al. showed that for any words $\bfu$ and $\bfv$
not ending with $\tx_0$ (representing either admissible or non-admissible MZVs)
\begin{equation*}
\zeta_\sha(\bfu)\zeta_\sha(\bfv)=\zeta_\sha(\bfu \sha \bfv)\quad\text{and} \quad
 \zeta_\ast(\bfu)\zeta_\ast(\bfv)=\zeta_\ast(\bfu \ast \bfv).
\end{equation*}
Moreover, there is an $\R$-linear map $\rho:\R[T]\to \R[T]$ defined by
$$
\rho(\exp(Tu))=\exp\Big( \sum_{n\ge 2} \frac{(-1)^n}{n} \zeta(n) u^n \Big)\exp(Tu)
$$
for a formal variable $u$, such that $\rho\circ \zeta_\ast =\zeta_\sha$. This relation is called the
\emph{regularized} (or \emph{extended}) DBSF.

As an application of the regularized DBSF we can now prove Euler's identity \eqref{equ:EulerId}. Indeed, from the stuffle product
of $H_M(1)H_M(2)$ we have
\begin{align*}
\zeta_\ast(1)\zeta(2)=\zeta_\ast(\tx_0\tx_1\ast \tx_1)
&\, =\zeta_\ast(\tx_1\tx_0\tx_1+\tx_0\tx_1\tx_1+\tx_0^2\tx_1)\\
&\, =\zeta_\ast(1,2)+\zeta(2,1)+\zeta(3).
\end{align*}
On the other hand, from the shuffle product of $\zeta^{(\eps)}(1)\zeta^{(\eps)}(2)$, we see that
$$
\zeta_\sha(1)\zeta(2)=\zeta_\sha(\tx_0\tx_1\sha\tx_1)=\zeta_\sha(\tx_1\tx_0\tx_1+2\tx_0\tx_1\tx_1)
=\zeta_\sha(1,2)+2\zeta(2,1).
$$
Hence the regularized DBSF implies that
\begin{equation*}
\zeta(2)T-2\zeta(2,1)=\zeta_\sha(1,2)=\rho(\zeta_\ast(1,2))=
\zeta(2)T-\zeta(2,1)-\zeta(3),
\end{equation*}
since $\rho(T)=T$. Thus Euler's identity $\zeta(2,1)=\zeta(3)$ follows immediately. Incidentally, we find that
the renormalized value of ``$\zeta(1,2)$'' is
$$
 \rint{0}{1}{-3.3ex}{1.5ex} \tx_1\tx_0\tx_1=\zeta_\sha(1,2)\big|_{T=0}= -2\zeta(2,1).
$$

\section{Multiple polylogarithm and colored MZVs}
\subsection{Motivation}
The study of MZVs has increasingly attracted both mathematicians and physicists, driven by recent advancements in computational power that enable the evaluation of the enormous integrals arising from complex Feynman diagrams. Remarkably, the outputs of these high-loop computations include not only MZVs but also special values of multiple polylogarithms. Fix $\bfk=(k_1,\dotsc,k_d)\in\N^d$. For complex variables $(z_1,\dotsc,z_d)\in\CC^d$ one defines the \emph{multiple polylogarithm} of \emph{weight} $|\bfk|:=k_1+\dotsm+ k_d$ and \emph{depth} $d$ by
\begin{equation}\label{equ:LiDefn}
\Li_\bfk(z_1,\dotsc,z_d):=
\sum\limits_{n_1>\cdots>n_d>0 } \frac{z_1^{n_1}\dotsm z_d^{n_d}}{n_1^{k_1}\cdots n_d^{k_d}},
\end{equation}
which converges if $|z_1\dotsm z_j|<1$ for all $j=1,\dots,d$.
For convenience, if $d=0$ we set $\bfk=\emptyset$ and $\Li_\emptyset=1$.
As in the logarithm case, the multiple polylogarithm \eqref{equ:LiDefn} can be analytically continued to a
multi-valued meromorphic function on $\CC^d$ (see \cite[Ch.\ 2]{Zhao2016}).

The name ``Li'', made standard in Lewin's book \cite{Lewin1981}, comes from the French phrase ``logarithmes it\'er\'es''
due to the fact that \eqref{equ:LiDefn} has an iterated integral representation. With an additional complex variable $z$, we have
\begin{equation}\label{equ-MPLitInt}
\Li_\bfk(zz_1,\dotsc,z_d)= \int_0^z \left(\frac{dt}{t}\right)^{k_1-1} \frac{dt}{u_1-t} \dotsm
\left(\frac{dt}{t}\right)^{k_d-1} \frac{dt}{u_d-t},
\end{equation}
where $u_j=(z_1\dotsm z_j)^{-1}$, if $z_1\dotsm z_d\ne 0$. This formula can be proved easily by induction on the weight $|\bfk|$ by differentiating with respect to $z$. Formula \eqref{equ-MPLitInt} provides an easy way to analytically continue the multiple polylogarithm.

Let $N,d\in\N$ and $\bfk=(k_1,\dotsc,k_d)\in\N^d$. For any $N$th roots of unity
$\eta_1,\dotsc,\eta_d$ one calls $\Li_\bfk(\eta_1,\dotsc,\eta_d)$
defined by \eqref{equ:LiDefn} a \emph{colored} (or \emph{cyclotomic}) \emph{multiple zeta value} (CMZV) of \emph{level} $N$,
which converges if and only if $(k_1,\eta_1)\ne(1,1)$ (see, e.g., \cite[Lemma 3.3.16]{Panzer2014c}).
These values at lower weights and depths appeared unexpectedly during the numerical computation of numerous Feynman integrals in the 1990s, when data from CERN and other international laboratories were first analyzed by mathematicians and theoretical physicists utilizing high-performance computers. We refer the interested reader to \cite[\S 14.7]{Zhao2016} for a fascinating history centered around this topic.

\subsection{Single variable multiple polylogarithm}
We have already seen that the single-variable multiple polylogarithm defined by \eqref{equ:singleVMPL} is key to proving the MZV duality relation. By setting $z_1 = \dots = z_d = 1$ in \eqref{equ-MPLitInt}, we obtain
\begin{equation}\label{equ-singleMPLitInt}
\Li_\bfk(z)= \int_0^z \left(\frac{dt}{t}\right)^{k_1-1} \frac{dt}{1-t} \dotsm
\left(\frac{dt}{t}\right)^{k_d-1} \frac{dt}{1-t}.
\end{equation}
To express succinctly the fact that only the $1$-forms $d\log t$ and $-d\log(1-t)$ are required, we denote the right-hand side of \eqref{equ-singleMPLitInt} by $\int_0^z [\om_0, \om_1]_w$, where $w=|\mathbf{k}|$ is the total number of such $1$-forms. This integral representation may then be used to define the analytic continuation of $\Li_\bfk(z)$.

The following lemma is crucial for understanding the asymptotic behavior of $\Li_\bfk(z)$ when $z\to \infty$. To state the result precisely, for a sequence of $s+1$ indices $\bfn_j\in\N^{d_j}$ (where $0 \le j \le s$), we define the tuple $\bfn_j\in\N^{d_j}$, with total weight $|\underline{\bfn}|=|\bfn_0|+\dotsm+|\bfn_s|$ and total depth $\dep(\underline{\bfn})=d_0+\dotsm+d_s$. Here, we permit $d_j = 0$, in which case $\bfn_j$ reduces to the empty index $\emptyset$, and we adopt the convention $\Li_{\emptyset}(z)=\zeta(\emptyset)=1$.

\begin{lem} \label{lem:Panzer}
For any $d\in\N$ and $\bfkk\in\N^d$, there are rational numbers $q(a,b,\underline{\bfnn})$ such that
\begin{equation}\label{equ:Panzer}
\Li_\bfkk(1/z)=(-1)^{|\bfkk|-d}\Li_\bfkk(z)+\sum_{a,b,\underline{\bfnn}} q(a,b,\underline{\bfnn}) (2\pi i)^{a} \log^{b}(-z) \Li_{\bfnn_0}(z) \prod_{j=1}^s \zeta(\bfnn_j),
\end{equation}
where $a,b,s$ range over nonnegative integers and $\underline{\bfnn}=(\bfnn_0,\dots,\bfnn_s)$
ranges over $s+1$ indices such that $a+b+|\underline{\bfnn}|=|\bfkk|$ and $\dep(\underline{\bfnn})<d$.
\end{lem}
\begin{proof} This follows from the proof of \cite[Theorem~1.3]{Panzer2017}.
\end{proof}

Motivated by Lemma \ref{lem:Panzer}, we may define the \emph{regularized value}
$\Li^\sha_\bfk(\infty;T)$ by replacing $\log(-z)$ by $T$ and setting
$\bfn_0=\emptyset$ in \eqref{equ:Panzer} since $\Li_{\bfn_0}(0)=0$ for all $\bfn_0\ne\emptyset$.
Moreover, we can now obtain its \emph{renormalized value} by setting $T=0$:
\begin{equation*}
\Li^\sha_\bfk(\infty)=\sum_{a,\bfn_1,\dotsc,\bfn_s} q(a,0,\emptyset,\bfn_1,\dotsc,\bfn_s) (2\pi i)^{a} \prod_{j=1}^s \zeta(\bfn_j).
\end{equation*}
In fact, by remaining within the field of real numbers when taking the limit, we find that $a$ must be even, allowing us to express the values in terms of $\zeta(2)$ rather than powers of $2\pi i$. Hence, we may rewrite the above as
\begin{equation} \label{equ:RenormMPLinfty}
\Li^\sha_\bfk(\infty):=\rint{0}{\infty}{-4.2ex}{2ex} \bfw(\bfk)
:=\sum_{2|a,\bfn_1,\dotsc,\bfn_s} q'(a,0,\emptyset,\bfn_1,\dotsc,\bfn_s) \zeta(2)^{a/2}\prod_{j=1}^s \zeta(\bfn_j)
\end{equation}
for some $q'\in\Q$. The example on the bottom of \cite[p. 96]{Panzer2017} can be written as
\begin{equation*}
\Li_{2,1}(1/z)=\Li_{3}(z)-\Li_{2,1}(z)+\zeta(3)-\log(-z)\Li_{2}(z)-\frac16 \log^3(-z).
\end{equation*}
Thus $\Li^\sha_{2,1}(\infty)=\zeta(3)$.

We will see in a moment that $\Li_\bfk(z)$ plays an important role in Au's recent paper \cite{Au2020} in which
he proved the following result concerning some Ap\'ery-like sums (see the next section).
Let $\CMZV_w^N$ be the $\Q$-span of all CMZVs of weight $w$ and level $N$.
\begin{thm}
The Ap\'ery-like sums
\begin{align*}
& \sum_{n=1}^\infty \frac{ 4^n H_{n-1}(\bfkk)}{n^l \binom{2n}{n}} \in\CMZV_{|\bfkk|+l}^2,\qquad
\sum_{n=1}^\infty \frac{ 2^n H_{n-1}(\bfkk)}{n^l\binom{2n}{n}} \in\CMZV_{|\bfkk|+l}^4.
\end{align*}
\end{thm}
The key idea in Au's proof is his observation that
\begin{equation}\label{equ:AuN=2}
\frac1{n \binom{2n}{n} }= \int_0^1 \frac{x^{n-1}}{(1+x)^{2n}}\, dx
\end{equation}
and that for the rational function $U_N(x)=8x/\big(N(1+x)^2\big)$, $N=2$ or 4, the integrals
\begin{equation*}
 \int_0^1 \frac1{x} \Li_\bfk( U_N(x) ) \, dx \in \CMZV_{|\bfk|+1}^N.
\end{equation*}
To prove this we need the concept of $N$-unital functions (see \S{}\ref{sec:Unital}).

\section{Ap\'ery-like sums}
In his celebrated proof of the irrationality of $\zeta(2)$ and $\zeta(3)$ published in 1978, Ap\'ery crucially relied on the following two identities:
\begin{equation}\label{equ:AperyZ2Z3}
\zeta(2) = 3\sum_{n\ge 1} \frac{1}{n^2\binom{2n}{n}}
\quad\text{and}\quad \zeta(3) = \frac{5}{2}\sum_{n\ge 1} \frac{(-1)^{n-1}}{n^3\binom{2n}{n}}.
\end{equation}
These formulas were subsequently extended to express other Riemann zeta values using series similar to those on the right-hand side of \eqref{equ:AperyZ2Z3} (see, e.g., \cite{AlmkvistGr1999, BorweinBr1997, Koecher1980}).

Also motivated by Ap\'ery's work, Leshchiner \cite{Leshchiner} generalized these identities to higher-weight (alternating) Riemann zeta values and their odd-indexed analogs. At arbitrary depth, one can define two types of generalizations depending on whether the central binomial coefficients appear in the numerator or the denominator. Although one could further introduce other binomial variants such as $\binom{3n}{n}$ and $\binom{4n}{2n}$, we do not consider them in the present paper.

If the central binomial coefficients appear in the numerator (resp.\ denominator), we call the sum an \emph{Ap\'ery-type binomial series} (resp.\ \emph{inverse binomial series}). For example, for any depth $d\in\N$, index vector
$\bfs=(s_1,\dots, s_d)\in\N^d$, and complex variable $z$, we can define a class of Ap\'ery-type inverse binomial series by
\begin{equation}\label{defn:gs}
\gs(\bfs;z):=\sum_{n_1> n_2>\cdots>n_d>0} {\binom{2n_1}{n_1}}^{-1} \frac{(2z)^{2n_1}}{(2n_1)^{s_1}\cdots (2n_d)^{s_d}}.
\end{equation}
Here we have renormalized the series so that its odd-indexed variation exhibits analogous behavior. By Stirling's formula,
\begin{equation*}
4^n \binom{2n}{n}^{-1} \sim \sqrt{\pi n} \quad \text{as } n \to \infty.
\end{equation*}
As before, we call $|\bfs| := s_1 + \dots + s_d$ the \emph{weight} and $d$ the \emph{depth} of the series $\gs(\bfs;z)$. It is easy to see that $\gs(\bfs;z)$ converges for all $|z| < 1$; moreover, if $\mathbf{s}$ is admissible, convergence holds on the boundary $|z| = 1$. In view of \eqref{equ:AperyZ2Z3}, one naturally wonders whether these series are related to CMZVs at specific values of $z$. More specifically, we aim to address the following question.

\medskip\noindent
{\bf Question 1.} Can we express $\gs(\bfs;z)$ and its alternating counterpart in terms of CMZVs for $z=\sqrt{j}/2$ for $j\in \{1, 2, 4\}$? \label{question}
\medskip

Another equally important motivation for studying Ap\'ery-like sums stems from theoretical physics. It turns out that these series, much like CMZVs, play significant roles in the evaluation of Feynman integrals. Numerous experimental studies emerged around the beginning of this century; see, for instance, \cite{BorweinBrKa2001, DavydychevDe2001, DavydychevDe2004} for inverse binomial series and \cite{JegerlehnerKV2003} for binomial series. In fact, odd variations of both types had already appeared implicitly in \cite[(1.1)]{DavydychevDe2004} and \cite[(A.25)]{JegerlehnerKV2003}, respectively. By ``odd variations,'' we mean the series obtained from \eqref{defn:gs} by shifting the even parity of certain indices from $2n_j$ to the odd counterparts $2n_j+1$.

Within the theoretical framework of MZVs and their alternating versions, investigating odd-indexed variations has proven to be a highly fruitful approach (see \cite{Hoffman2019, KanekoTs2019, XuZhao2020a}). These variations correspond essentially to level-two CMZVs.

\medskip\noindent
{\bf Question 2.} Can we express odd variations of $\gs(\bfs;z)$ and its alternating counterpart in terms of CMZVs? \label{question2}
\medskip

Diverse methodologies have been developed to evaluate Ap\'ery-like sums and their odd variants. Ablinger \cite{Ablinger2015} first converts various Ap\'ery-type inverse binomial series into iterated integrals, which are subsequently computed using cyclotomic harmonic polylogarithms. Chu \cite{Chu2021a, Chu2021b} and his collaborators have applied hypergeometric series transformations to express elegant $\Q$-linear combinations of these sums in terms of Riemann zeta values, thereby proving several conjectures of Z.-W.~Sun. However, this approach has been restricted exclusively to Ap\'ery-type inverse binomial series. Alternatively, Campbell and his collaborators \cite{CampbellCA2022, CantariniD2019, CampbellDS2019} leverage Fourier--Legendre series and the properties of the complete elliptic integral of the first kind to derive compact formulas for Ap\'ery-type binomial series, though they obtain few inverse ones. Furthermore, multiple integrals associated with $N$-posets can also be utilized to establish relations among Ap\'ery-like sums \cite{XuZhao2022May}.

In the remainder of this paper, we explore five additional methods for evaluating these sums. The primary tools utilized are: (i) single-variable multiple polylogarithms evaluated at unit-root arguments, (ii) Clausen functions, (iii) trigonometric iterated integrals, (iv) the fibration method for iterated integrals (which can be viewed as an extension of the first method), and (v) contour integration.

\section{Ap\'ery-like sums via $N$-unital functions}\label{sec:Unital}
\subsection{$N$-unital functions}
Let $N$ be a positive integer. We say a non-constant rational function $U(x)\in\CC(x)$ is $N$-\emph{unital}
if every zero and pole of both $U(x)$
and $1-U(x)$ is either 0 or an $N$th roots of unity. It follows from \cite{Mason1983} or \cite[Cor. 2.2]{Brunault2021}
that there are only finitely many $N$-unital functions for any fixed $N\in\N$. Let $\vU_N$ be the set of all such functions.
For example, it is not hard to show that $\vU_1=\langle x\rangle_6$ where for each function $f$, $\langle f\rangle_6$
is its $S_3$-orbit defined by
$$
\langle f\rangle_6:=\left\{ f, 1-f,\frac{1}{f},\frac{f}{f-1}, \frac{f-1}f, \frac{-1}{f-1}\right\}.
$$

In fact, we are more interested in the set of all the possible values of $U(0)$.
For small levels this is given by the next result.

\begin{prop} \emph{(\cite{Zhao-4adm})} \label{prop:R0}
Let $N\le 4$ and $\mu_j=\exp(2j\pi i/N)$ for $1\le j<N$. Define
\begin{equation} \label{equ:ConjR0}
 \calC^N:=\{0,1,\infty\} \cup
\left\{
 \begin{array}{ll}
 \displaystyle \bigcup_{j=1}^{\frac{N}{2}} \langle \mu_j \rangle_6, & \hbox{if $N$ is even;} \\
 \displaystyle \bigcup_{\eps=\pm1} \bigcup_{j=1}^{\frac{N-1}{2}} \langle \eps \mu_j \rangle_6, \qquad \ & \hbox{if $N$ is odd.}
 \end{array}
\right.
\end{equation}
Then $U(0)\in \calC^N$ for all $U(x)\in \vU_N$.
\end{prop}

\begin{proof} The proof is by an explicit computation of $\vU_N$ via a tedious albeit elementary process.
Precisely, there are $6$ $1$-unital, $36$ $2$-unital, $84$ $3$-unital, and 252 $4$-unital functions.
\end{proof}

\begin{re}\label{re-5Unital}
When $N\le 4$ we verified the unital functions using the MAGMA program kindly shared with us by F.\ Brunault.
The program shows that there are 336 $5$-unital functions.
Unfortunately, even for $N=5$ the set $\calC^5$ in the proposition is too small to contain all values of 5-unital functions at 0.
Concretely, setting $\mu=\mu_5=\exp(2\pi i/5)$ we have found the following exceptional 5-unital functions:
\begin{align*}
f(x)=&\frac{(\mu^3 + \mu^2 + 2)x^2 + (\mu^4 + \mu^2)x - \mu^2 + \mu - 1}{x^2-(\mu^4 + \mu^2)x + \mu},\\
g(x)=&\frac{(\mu^2 + \mu + 1)x - \mu^2 - \mu - 1}{x-\mu^2}. %-\mu^2 - \mu
\end{align*}
One can easily check that
$$
f(x)=\frac{(\mu^3 + \mu^2 + 2)(x- 1)(x-\mu)}{(x-\mu^2)(x-\mu^4)}, \quad
g(x)=\frac{(\mu^2 + \mu + 1)(x-1)}{x-\mu^2}
\in\vU_5
$$
since
$$
1-f(x)=-\frac{(\mu^3 + \mu^2 + 1)(x-\mu^3)^2}{(x-\mu^2)(x-\mu^4)}, \quad
1-g(x)=\frac{-(\mu^2 + \mu)(x-\mu^4)}{x-\mu^2}
\in\vU_5.
$$
But none of the following values are in $\calC^5$ defined by \eqref{equ:ConjR0}:
\begin{alignat*}{6}
\langle f(0)\rangle_6=&\, \Big\{\frac{3-\sqrt{5}}2, \frac{\sqrt{5}-1}2, \frac{3+\sqrt{5}}2,
-\frac{1+\sqrt{5}}2, \frac{1-\sqrt{5}}2, \frac{\sqrt{5}+1}2\Big\}, \\
\langle g(0)\rangle_6=&\, \Big\{\mu+\mu^2, 1-\mu-\mu^2,-\mu-\mu^2, 1+\mu+\mu^2,-\mu^3-\mu^4 , 1+\mu^3+\mu^4 \Big\}.
\end{alignat*}
\end{re}

\subsection{Iterated integral involving $N$-unital functions}
We will see that the following conjecture due to Au is crucial in the evaluation of Ap\'ery-like sums via unital functions.
\begin{conj}\label{conj:MPLunital}
Let $\bfkk\in\N^d$. For all $U(x)\in\vU_N$, the renormalized value
\begin{equation*}
 \rint{0}{1}{-3.3ex}{1.3ex} \frac1{x} \Li_\bfkk( U(x) ) \, dx \in \CMZV_{|\bfkk|+1}^N.
\end{equation*}
\end{conj}

First, we should explain how the regularization is defined in this context. We start by considering
\begin{equation*}
 \int_{\eta}^{1-\eps} \frac1{x} \Li_\bfk( U(x) ) \, dx
\end{equation*}
for arbitrarily small numbers $\eta,\eps>0$. The main philosophy is that under ideal situations like ours, we may rewrite the
integrand as a finite sum of ``well-behaved'' functions so that the integral of each part can be expressed as
a polynomial of $\log \eta$ and $\log \eps$ modulo $O(\eta\log^J \eta)$ and $O(\eps\log^J \eps)$ where $J=\dep(\bfk)+1$.
This process, which is completely similar
to that for obtaining the regularized DBSF of MZVs, is again called \emph{regularization}.
Although each summand may diverge as $\eta,\eps\to 0$, we can just consider the constant term of each summand
since we know the total sum itself has finite limit.
This constant term is then called the \emph{renormalization} of the corresponding summand, completely
similar to the renormalizations of \eqref{equ:renormMZV} and \eqref{equ:RenormMPLinfty}.
We again denote this regularized and then renormalized value by
$
\displaystyle \rint{0}{1}{-3.3ex}{1.3ex} \frac1{x} \Li_\bfk( U(x) ) \, dx.
$

Using the notation from the last section, if $U(x)\in\vU_N$ then  by \eqref{equ-singleMPLitInt}
\begin{equation*}
\Li_\bfk( U(x) )=\int_0^{U(x)} \gf_1\dotsm \gf_w,
\end{equation*}
where $w=|\bfk|$ is the weight and the 1-forms $\gf_j$ are either $d\log t$ or $-d\log(1-t)$.
Then Lemma~\ref{lem:chen's}(i) immediately yields that
\begin{equation}\label{equ:2ndFac}
\rint{0}{1}{-3.3ex}{1.3ex} \frac1x \Li_\bfk( U(x) ) \, dx=\sum_{j=1}^w
\left( \rint{0}{1}{-3.3ex}{1.3ex} \frac1x \int_{U(0)}^{U(x)}\gf_1\dotsm \gf_j \, dx \right)
\left(\rint{0}{U(0)}{-6.2ex}{3.6ex}\gf_{j+1} \dotsm \gf_w \right).
\end{equation}
Here we have removed the $j=0$ term since it vanishes after regularization and renormalization.

It is clear to see that under the change of variables $t\to U(t)$ the first factor in the products
above is reduced to an iterated integral over $[0,x]$ of 1-forms of the shape $d\log U(x)$ and $d\log (1-U(x))$.
Then the renormalized value
\begin{equation*}
\rint{0}{1}{-3.3ex}{1.3ex} \frac{dx}{x} \int_{U(0)}^{U(x)} [d\log t,d \log (1-t)]_j
 = \rint{0}{1}{-3.3ex}{1.3ex} \frac1{x} \int_0^x [d\log U(x), d\log (1-U(x))]_j
\end{equation*}
must lie in $\CMZV_{j+1}^N$ by the definition of $N$-unitality.

However, it is much trickier to treat the second factor in \eqref{equ:2ndFac},
which amounts to finding the set $\vU_N$ of all possible $N$-unital functions $U(x)$
and then prove that the iterated integrals $\int_0^{U(0)}\gf_{j+1} \dotsm \gf_w$ lie in the correct space.

\begin{thm}\label{thm-NUnitalU(0)}
Let $\calC^N$ be defined by \eqref{equ:ConjR0}. For any index vector $\bfk$ of positive integers and $U(x)\in\vU_N$,
if $U(0)\in\calC^N$ then
\begin{equation*}
 \rint{0}{1}{-3.3ex}{1.3ex} \frac1{x} \Li_\bfkk( U(x) ) \, dx \in \CMZV_{|\bfkk|+1}^N.
\end{equation*}
\end{thm}
\begin{proof}
By \eqref{equ:2ndFac} it suffices to prove the claim that for all weight $w$ and $N$-unital function $U(x)$ with $U(0)\in\calC^N$,
\begin{equation*}
\rint{0}{U(0)}{-6.2ex}{3.6ex} [d\log t, d\log(1-t)]_w\in \CMZV_w^N.
\end{equation*}
If $U(0)=0$ or $1$ then the claim is trivial. If $U(0)=\infty$, then the claim follows immediately from \eqref{equ:RenormMPLinfty}.
This complete the proof of $N=1$ case.

Assume $N\ge 2$. For $U(0)\ne 0,1,\infty$ since
$\mu_j^{-1}$ is still an $N$th root of unity, for simplicity,
we may just assume $\mu=\mu_j$ for some $j=1,\dots, N-1$.

\begin{enumerate}
 \item[\upshape{(1)}]
$U(0)=\mu$. The claim follows quickly from the substitution $t=\mu t$.

 \item[\upshape{(2)}] \label{U(0)=1-u}
$U(0)=1-\mu$. Applying Lemma~\ref{lem:chen's}(i) again we get
\begin{equation*}
\rint{0}{1-\mu}{-5.9ex}{3.3ex} [\tx_0, \tx_1]_w
=\sum_{j=0}^w \rint{1}{1-\mu}{-5.9ex}{3.3ex} [\tx_0, \tx_1]_j \cdot \rint{0}{1}{-3.3ex}{1.3ex} [\tx_0, \tx_1]_{w-j} .
\end{equation*}
Then by change of variables $t\to 1-t$ (so that $\tx_0 \leftrightarrow \tx_1$) in the first factor we reduce it to Case 1)
by Lemma~\ref{lem:chen's}(ii) (notice the order of integration is reversed by this change of variables).

 \item[\upshape{(3)}]
$U(0)=\frac1{1-\mu}$. Applying Lemma~\ref{lem:chen's}(i) again we get
\begin{equation*}
\rint{0}{\frac1{1-\mu}}{-5.95ex}{3.55ex} [\tx_0, \tx_1]_w
=\sum_{j=0}^w \rint{1}{\frac1{1-\mu}}{-5.95ex}{3.55ex} [\tx_0, \tx_1]_j \cdot \rint{0}{1}{-3.3ex}{1.3ex} [\tx_0, \tx_1]_{w-j} .
\end{equation*}
Then by change of variable $t\to 1/t$ (so that $\tx_0\to -\tx_0, \tx_1\to \tx_1-\tx_0$)
in the first factor we get
\begin{equation*}
\rint{1}{\frac1{1-\mu}}{-5.95ex}{3.55ex} [\tx_0, \tx_1]_j = \rint{1}{1-\mu}{-5.85ex}{3.25ex} [-\tx_0,\tx_1-\tx_0]_j\in \CMZV_j^N
\end{equation*}
by Case (2) and Lemma~\ref{lem:chen's}(ii).

 \item[\upshape{(4)}]
$U(0)=-\mu\ne -1$. Put $\tx_\mu=dt/(\mu-t)$. Then
\begin{equation*}
\rint{0}{-\mu}{-4.95ex}{2.45ex} [\tx_0,\tx_1]_w
=\sum_{j=0}^w \rint{1}{-\mu}{-4.9ex}{2.4ex} [\tx_0, \tx_1]_j \cdot \rint{0}{1}{-3.3ex}{1.3ex} [\tx_0, \tx_1]_{w-j} .
\end{equation*}
For the first factor, by substitution $t=\frac{t-\mu}{1-\mu t}$ we see that $\tx_0\to \tx_\mu- \tx_\bmu$
and $\tx_1\to \tx_1- \tx_\bmu$. Thus by Lemma~\ref{lem:chen's}(ii)
\begin{align*}
\rint{1}{-\mu}{-4.95ex}{2.45ex} [\tx_0,\tx_1]_w= &\, \rint{1}{0}{-3.35ex}{1.35ex} [\tx_\mu- \tx_\bmu, \tx_1- \tx_\bmu]_w \\
= &\, (-1)^w\rint{0}{1}{-3.3ex}{1.3ex} \tau[\tx_\mu- \tx_\bmu, \tx_1- \tx_\bmu]_w\in \CMZV_w^N,
\end{align*}
where $\tau$ means the order of the iterated integral is reversed.

 \item[\upshape{(5)}] \label{U(0)=1+u}
$U(0)=1+\mu$. By substitution $t=\frac{(1+\mu)(1-t)}{1-\mu t}$ we see that $\tx_0\to \tx_1- \tx_\bmu$
and $\tx_1\to \tx_\mu- \tx_\bmu$. Thus by Lemma~\ref{lem:chen's}(ii)
\begin{align*}
\rint{0}{1+\mu}{-5.8ex}{3.4ex} [\tx_0,\tx_1]_w= &\,\rint{1}{0}{-3.3ex}{1.3ex} [\tx_1- \tx_\bmu, \tx_\mu- \tx_\bmu]_w\\
= &\,(-1)^w\rint{0}{1}{-3.3ex}{1.3ex} \tau[\tx_1- \tx_\bmu, \tx_\mu-\tx_\bmu]_w\in \CMZV_w^N.
\end{align*}

 \item[\upshape{(6)}]
$U(0)=\frac1{1+\mu}$. Applying Lemma~\ref{lem:chen's}(i) again we get
\begin{equation*}
\rint{0}{\frac1{1+\mu}}{-5.9ex}{3.5ex} [\tx_0, \tx_1]_w
=\sum_{j=0}^w \rint{1}{\frac1{1+\mu}}{-5.9ex}{3.5ex} [\tx_0, \tx_1]_j \cdot \rint{0}{1}{-3.3ex}{1.3ex} [\tx_0, \tx_1]_{w-j} .
\end{equation*}
Then by change of variable $t\to 1/t$ (so that $\tx_0\to -\tx_0, \tx_1\to \tx_1-\tx_0$),
in the first factor we get
\begin{equation*}
\rint{1}{\frac1{1+\mu}}{-5.9ex}{3.5ex} [\tx_0, \tx_1]_j = \rint{1}{1+\mu}{-5.9ex}{3.4ex} [-\tx_0,\tx_1-\tx_0]_j \in \CMZV_j^N
\end{equation*}
by Case (5) and Lemma~\ref{lem:chen's}(ii).
\end{enumerate}
This completes the proof of the theorem.
\end{proof}

\begin{re}
In Remark~\ref{re-5Unital} we have found $f(x),\,g(x) \in\vU_5 $ such that
$\langle f(0)\rangle_6\cap \calC^5=\langle g(0)\rangle_6\cap \calC^5=\emptyset$.
Au's computation in \cite[Table 1]{Au2022} does
contain the three values $f(0), 1-f(0), f(0)/(f(0)-1)$ but not the others, nor any values in $\langle g(0)\rangle_6$.
These are essentially the only exceptions one need to check to prove Theorem~\ref{thm-NUnitalU(0)} for $N=5$.
\end{re}

\subsection{Evaluation of Ap\'ery-like sums via $N$-unital functions}
As Ap\'ery demonstrated, central binomial series such as \eqref{equ:AperySeries} can play a vital role in number theory. Consequently, numerous studies on this topic have appeared in recent years, building upon a long history of classical investigation. In this section, utilizing a key observation by Au \cite{Au2020} that relates these series to CMZVs via unital functions, we present a family of such series for each level $N\ge 2$.

Fix $N\ge 2$ and $1\le j\le N-1$. Put $\theta=\theta_j^{(N)}:=2j\pi/N$, $\mu=\mu_j^{(N)}:=e^{i\theta}$, and
$a=\mu+\bmu=2\cos\theta$. Define
\begin{align*}
& f_n=f_n^{(N,j)}:=\int_0^1 \frac1x \left(\frac{(2-a)x}{x^2-a x+1}\right)^n \, dx , \\
& g_n=g_n^{(N,j)}:=\int_0^1 \left(\frac{(2-a)x}{x^2-a x+1}\right)^n \, dx.
\end{align*}
By direct computation we see that
\begin{align*}
f_1=& \frac{(2-a)}{2\sin\theta} (\pi-\theta)=(\pi-\theta)\tan \frac{\theta}2 , \quad
f_2=\frac{2-a}{4\sin^2 \theta}(af_1+2-a),\\
g_0=1, \quad g_1=& (2-a) \log\Big(2\sin \frac{\theta}{2} \Big)+\frac{a(2-a)}{4\sin\theta} (\pi-\theta), \quad
g_2= \frac1{a}(2f_2-2+a).
\end{align*}
The only exception is when $N=2$ and $j=1$ in which case we can compute directly that $g_1=4\log 2-2$.
Assuming $n\ge 1$, by integration by parts we have
\begin{align*}
 f_n=&\, \left.\frac{(2-a)^n x^n}{n(x^2-ax+1)^n}\right|_0^1+\int_0^1 \frac{(2-a)^n x^n(2x-a)}{(x^2-ax+1)^{n+1}} \, dx \\
=&\, \frac{1}{n} +\frac{2}{2-a} g_{n+1}-\frac{a}{2-a} f_{n+1}.
\end{align*}
Similarly,
\begin{align*}
 g_n=&\, \left.\frac{(2-a)^n x^{n+1}}{(n+1)(x^2-ax+1)^n}\right|_0^1+\frac{n}{n+1}\int_0^1 \frac{(2-a)^n x^{n+1}(2x-a)}{(x^2-ax+1)^{n+1}} \, dx \\
=&\, \frac{1}{n+1} +\frac{n}{n+1}\int_0^1 \frac{ (2-a)^n[ 2\cdot(x^{n+2}-ax^{n+1}+x^n)+a x^{n+1}-2 x^n]}{(x^2-ax+1)^{n+1}}\, dx\\
=&\, \frac{1}{n+1} +\frac{n}{n+1}\left(2g_n+\frac{a}{2-a}g_{n+1}-\frac{2}{2-a} f_{n+1}\right).
\end{align*}
Therefore we get the second order recurrence for both $\{f_n\}$ and $\{g_n\}$: for all $n\ge 2$
\begin{align*}
(a+2)nf_{n+1}-a(2n-1) f_{n}-(2-a)(n-1)f_{n-1}=0, \\
(a+2)n(n-1)g_{n+1}-a(2n-1)(n-1) g_{n}-(2-a)n(n-2)g_{n-1}=& 2-a.
\end{align*}

\begin{thm} \label{thm:sumGn}
Let $N\ge 2$, $1\le j<N$ and $\theta=2j\pi/N$. Set
\begin{equation*}
F_1=\left\{
 \begin{array}{ll}
 1 , & \hbox{if $N=2$;} \\
 (\pi-\theta)\csc \theta,\quad\ & \hbox{if $N\ge 3$,}
 \end{array}
 \right.
\quad
F_2=\left\{
 \begin{array}{ll}
 4/3 , & \hbox{if $N=2$;} \\
 (\cos\theta F_1+1) \csc^2\theta ,\quad\ & \hbox{if $N\ge 3$,}
 \end{array}
 \right.
\end{equation*}
and
\begin{equation*}
n \sin^2\theta F_{n+1}-(2n-1) \cos\theta F_{n}-(n-1) F_{n-1}=0 \quad \forall n\ge2.
\end{equation*}
Then we have
\begin{align*}
\sum_{n\ge1} \frac{H_{n-1}(\bfkk)}{n^l} (1-\cos\theta)^n F_n \in \CMZV_{|\bfkk|+l+1}^N.
 \end{align*}
In particular, when $l=1$ and $2$ we have
\begin{align*}
\sum_{n\ge1} \frac{(1-\cos\theta)^n F_n}{n} =&\, 2\zeta(2)-\Li_2(\mu)-\Li_2(\bmu)\in \CMZV_{2}^N, \\
\sum_{n\ge1} \frac{(1-\cos\theta)^n F_n}{n^2}=&\,\sum_{\eta=\mu,\bmu} \Big(2\Li_{2,1}(\eta,\bar\eta)-\Li_{2,1}(\eta,1)-\Li_{2,1}(\eta,\bar\eta^2)\Big)\\
&\,+2\zeta(3)-\Li_3(\mu)-\Li_3(\bmu)\in \CMZV_{3}^N.
\end{align*}
\end{thm}
\begin{proof}
It is easy to see that $f_n=(1-\cos\theta)^n F_n$. Put
$$
U(x):=\frac{(2-a)x}{x^2-a x+1}=\frac{(2-a)x}{(x-\mu)(x-\bmu)}, \quad 1-U(x)=\frac{(x-1)^2}{(x-\mu)(x-\bmu)}\in \vU_N.
$$
We see that
\begin{align*}
\sum_{n\ge1} \frac{H_{n-1}(\bfk)}{n^l} f_n
=& \int_0^1 \frac1x\sum_{n\ge1} \frac{H_{n-1}(\bfk)}{n^l} \big(U(x)\big)^n \, dx , \\
=& \int_0^1 \frac1x \Li_{l,\bfk}\big(U(x)\big) \, dx\\
=& \int_0^1 \frac1x \left(\int_0^{U(x)} [\tx_0,\tx_1]_{|\bfk|+l} \, dt \right)\,dx\\
=& \int_0^1 \tx_0 [ \tx_0+\tx_{\mu}+\tx_{\bmu},2\tx_1-\tx_{\mu}-\tx_{\bmu}]_{|\bfk|+l}\in \CMZV_{|\bfk|+l}^N
\end{align*}
since
\begin{align*}
d\log U(t)=&\, d\log t-d\log(\mu-t)-d\log(\bmu-t),\\
d\log (1-U(t))=&\, 2d\log(1-t)-d\log(\mu-t)-d\log(\bmu-t).
\end{align*}
In particular, when $\bfk=\emptyset$ we have
\begin{align*}
\sum_{n\ge1} \frac{f_n}{n^l}=&\int_0^1 \tx_0 (\tx_0+\tx_{\mu}+\tx_{\bmu})^{l-1}(2\tx_1-\tx_{\mu}-\tx_{\bmu}) .
\end{align*}
This implies quickly the two special cases when $l=1$ and $l=2$.
\end{proof}

\begin{cor}
For all $l\ge 2$ we have
\begin{align*}
\sum_{n\ge1} \frac{4^n H_{n-1}(\bfkk)}{n^l\binom{2n}{n}} \in  \CMZV_{|\bfkk|+l+1}^2,\quad
\sum_{n\ge1} \frac{\pi\binom{2n}{n}}{4^n\cdot (2n+1)^l}  \in  \CMZV_{l+1}^4.
 \end{align*}
\end{cor}

In the special case $\bfk=\emptyset$, the first sum has been widely investigated in the literature. We refer the reader to \cite[\S 3]{WangXu2021} for a brief historical overview and a Maple package dedicated to its evaluation. Additionally, a wide array of similar Ap\'ery-like sum evaluations are available in \cite{Au2020}.

\begin{proof}
The first claim, as treated in \cite{Au2020}, follows from \eqref{equ:AuN=2} after taking $N=2,j=1$ in the theorem.

For the second sum, set $N=4, j=1$, $u=i$ and $a=0$. Using the Gamma integral we may compute $f_n$ as follows:
\begin{align}
f_n =\int_0^1 \frac1x\left(\frac{2x}{x^2+1}\right)^n \, dx
= & 2^n \int_0^{\pi/4} \frac{\tan^{n-1}t}{\sec^{2n-2}t} \, dt \notag\\
= & 2^{n-1} \int_0^{\pi/2} \sin^{n-1}t\cdot \cos^{n-1} t \, dt
=2^{n-2} \frac{\gG(\frac{n}{2})\gG(\frac{n}{2}) }{\gG(n)}.   \label{equ:fn}
\end{align}
Hence
\begin{align*}
 f_{2k} =&\, 2^{2k-2}\frac{(k-1)!^2}{(2k-1)!}=\frac{4^k}{2k{\binom{2k}{k}}} ,\quad
 f_{2k+1} = \frac{(2k-1)!!^2}{2(2k)!}\pi = \frac{\pi\binom{2k}{k}}{2\cdot 4^k} .
\end{align*}
We see that
\begin{align*}
\sum_{n\ge1} \frac{f_n}{n^l}=& \sum_{k\ge1} \frac{4^k }{(2k)^{l+1}{\binom{2k}{k}}}
+ \sum_{k\ge 0} \frac{\pi\binom{2k}{k}}{2(2k+1)^l \cdot 4^k}  \in \CMZV_{l+1}^4.
\end{align*}
Thus the second claim follows immediately from the first.
\end{proof}

\begin{re} One can show that
\begin{equation} \label{equ:further}
\sum_{n\ge1} \frac{\binom{2n}{n}}{4^n\cdot (2n+1)^l}  \in i \CMZV_{|\bfk|+l}^4
\end{equation}
by the following identity
\begin{equation*}
 \sum_{n\ge1} \frac{\binom{2n}{n}}{4^n\cdot (2n+1)^l}  =
\sum_{n\ge1} \frac{(-1)^n \binom{-1/2}{n}}{(2n+1)^l} = \int_0^1 \left(\frac{dt}{t}\right)^{l-1} \frac{dt}{\sqrt{1-t^2} }.
\end{equation*}
With the substitution $t\to 2t/(1+t^2)$, we get $\tx_0\to \tx_0+\tx_i+\tx_{-i}$ and $ \frac{dt}{\sqrt{1-t^2}} \to
\frac{2dt}{1+t^2}=i(\tx_i-\tx_{-i}).$ Thus \eqref{equ:further} holds.
\end{re}

\begin{ex}
By taking $N=2,j=1$ and using \eqref{equ:AuN=2}, as in \cite{Au2020}, we can find
\begin{equation}\label{equ:N2}
\sum_{k\ge1} \frac{4^k}{k^{l+1}\binom{2k}{k}} =
\left\{
 \begin{array}{ll}
 \displaystyle 3 \zeta(2) , & \hbox{if $l=1$;} \\
 \displaystyle 6\zeta(2)\log 2 -\tfrac7{2}\zeta(3), \quad \ & \hbox{if $l=2$.}
 \end{array}
\right.
\end{equation}
This is consistent with \cite[Example 8]{CoppoCa2015}, \cite[Theorem 3.1]{Sprugnoli2006}, and \cite[Example 3.1]{WangXu2021}.
\end{ex}

\begin{ex}\label{cor:N=4}
Take $N=4$. If $l=1$, by Theorem~\ref{thm:sumGn} and a straight-forward calculation we get
\begin{align}
\sum_{n\ge1} \frac{f_n }{n} =& 2\zeta(2)-\Li_2(i)-\Li_2(-i)=\tfrac94\zeta(2), \label{N=4l=1}\\
\sum_{n\ge1} \frac{f_n}{n^2}=&2\zeta(3)-\Li_3(i)-\Li_3(-i)+\sum_{\eta=\pm i} \Big(2\Li_{2,1}(\eta,\bar\eta)-\Li_{2,1}(\eta,1)-\Li_{2,1}(\eta,\bar\eta^2)\Big) \notag \\
=& \tfrac{21}{16}\zeta(3)+\tfrac34\zeta(2)\log 2 + M(\bar2,\bar1) \notag
\end{align}
by manipulating the summation indices suitably, where
$$
M(\bar2,\bar1)=\sum_{m>n>0} \frac{4(-1)^{m+n}}{(2m-1)^2(2n-1)}=\tfrac32\zeta(2)\log 2-\tfrac74\zeta(3)
$$
is a double mixed value studied in \cite{XuZhao2020a}. To show the second equality in \eqref{N=4l=1} we have
\begin{equation*}
\Li_2(i)+\Li_2(-i)=\sum_{m\ge 1} \frac{i^m+(-i)^m}{m^2}=\sum_{m\ge 1} \frac{2(-1)^m}{(2m)^2}=\tfrac12\zeta(\bar2)=-\tfrac14\zeta(2),
\end{equation*}
where the evaluation of the alternating double zeta value $\zeta(\bar2):=\sum_{m\ge1}\frac{(-1)^m}{m^2}$ can be found easily by the well-known relation $\zeta(2)+\zeta(\bar2)=\tfrac12 \zeta(2)$. By \eqref{equ:N2} this implies
\begin{equation*}%\label{equ:2ndCor}
%\begin{split}
 \sum_{k\ge 0} \frac{\binom{2k}{k}}{(2k+1) \cdot 4^k}  = \frac{\pi}{2},\quad
 \sum_{k\ge 0} \frac{\pi \binom{2k}{k}}{2(2k+1)^2 \cdot 4^k}= \tfrac74\zeta(3)+M(\bar2,\bar1).
%\end{split}
\end{equation*}
Hence
\begin{equation*}
 \sum_{k\ge 0} \frac{\binom{2k}{k}}{(2k+1)^2 \cdot 4^k} =\frac{3}{\pi}\zeta(2)\log 2=\frac{\pi\log 2}{2}.
\end{equation*}
\end{ex}

\section{Ap\'ery-like sums via Clausen functions}
\subsection{Some conjectures of Z.-W.\ Sun}
Since the late 1980s, Z.-W.~Sun has discovered a multitude of elegant and highly nontrivial infinite series identities, frequently involving arithmetic objects such as Fibonacci, Bernoulli, and Euler numbers. In \cite{Sun2015, Sun2021, Sun2022}, he proposed several families of conjectures that generalize Lucas numbers to Lucas sequences. We highly recommend his recent book \cite{Sun2021}; although written in Chinese, the mathematical text is accessible, and readers can readily understand most of the 820 conjectures without any prior knowledge of the language.

Let $A$ and $B$ be integers. Sun defines a generalized \emph{Lucas sequence} $\{v_n(A,B)\}_{n\ge 0}$ as follows (see \cite{Sun2009}):
\begin{align*}
& v_0(A,B):=2,\quad v_1(A,B):=A,\\
& v_{n+1}(A,B)=Av_{n}(A,B)-Bv_{n-1}(A,B) \quad \forall n\ge 1.
\end{align*}
The characteristic equation $x^2-Ax+B=0$ of the sequence $\{v_n\}_{n\geq 0}$ has two roots
\[\alpha=\frac{A+\sqrt{A^2-4B}}{2}\quad \text{and}\quad \beta=\frac{A-\sqrt{A^2-4B}}{2}.\]
It is well known that for any integer $n\ge 0$ we have
\begin{align*}
v_n(A,B)=\alpha^n+\beta^n.
\end{align*}
For examples, setting $\phi=\frac{\sqrt{5}+1}{2}$ we have
\begin{align*}
&v_n(3,1)=L_{2n}=\phi^{2n}+\phi^{-2n},\\
&v_n(4,1)=(2-\sqrt{3})^n+(2+\sqrt{3})^n,\\
&v_n(4,2)=(2-\sqrt{2})^n+(2+\sqrt{2})^n,\\
&v_n(5,5)=(\sqrt{5}\phi)^n+(\sqrt{5}\phi^{-1})^n.
\end{align*}

Define $H_0:=0$ and the \emph{classical harmonic numbers}
\[
H_n:=H_n(1)=\sum_{k=1}^n \frac{1}{k}\quad\forall n\ge 1.
\]
In this section, we will prove the following results.
\begin{thm}\label{thm-sun's-conj} We have
\begin{align} \label{equ:10.62}
&\sum_{n=1}^\infty \frac{L_{2n}}{n^2\binom{2n}{n}}\left(H_{2n}-H_{n-1}\right)=\frac{41\ze(3)+4\pi^2\log(\phi)}{25},\\
&\sum_{n=1}^\infty \frac{v_n(5,5)}{n^2\binom{2n}{n}}\left(H_{2n}-H_{n-1}\right)=\frac{124\ze(3)+\pi^2\log\left(5^5\phi^6\right)}{50},\label{equ:10.63}\\
&\sum_{n=1}^\infty \frac{v_n(4,1)}{n^2\binom{2n}{n}}\left(H_{2n}-H_{n-1}\right)=\frac{23\ze(3)+2\pi^2\log(2+\sqrt{3})}{12},\label{equ:2022a}\\
&\sum_{n=1}^\infty \frac{v_n(4,2)}{n^2\binom{2n}{n}}\left(H_{2n}-H_{n-1}\right)=\frac{259\ze(3)+2\pi^2(5\log(2)+8\log(1+\sqrt{2}))}{128}.\label{equ:2022b}
\end{align}
\end{thm}

Motivated by his own study of the congruences in \cite{Sun2009}, Sun proposed
identities \eqref{equ:10.62} and \eqref{equ:10.63} as \cite[Conjectures 10.62 and 10.63]{Sun2021}
(see also \cite[Eqs. (3.11) and (3.12)]{Sun2015}). He further conjectured \eqref{equ:2022a}
and \eqref{equ:2022b} in \cite{Sun2022}.

\subsection{Proof of Theorem~\ref{thm-sun's-conj} by Clausen functions}
Recall that for any positive integer $m$ the Clausen function $\cl_m(\theta)$ is defined by
\begin{align*}
&{\cl}_{2m-1}(\theta):=\sum_{n=1}^\infty \frac{\cos(n\theta)}{n^{2m-1}}\quad \text{and}\quad {\cl}_{2m}(\theta):=\sum_{n=1}^\infty \frac{\sin(n\theta)}{n^{2m}} \qquad \forall \theta\in [0,\pi].
\end{align*}
Many special values of the Clausen function are closely related to the Riemann zeta values. For example, $\cl_{2m-1}(0)=\zeta(2m-1)$,
\begin{equation}\label{equ:Riemmann}
\cl_{2m-1}(\pi)=(2^{2-2m}-1)\zeta(2m-1),\quad\text{and}\quad \cl_{2m-1}\Big(\frac\pi2\Big)=\frac{\cl_{2m-1}(\pi)}{2^{2m-1}}.
\end{equation}
The following relations among special values of the Clausen function are well known (see, e.g., \cite[Eq. (6.51)]{Lewin1981}):
for any positive integer $r$ we have
\begin{align}
&\cl_3\big(\tfrac{\pi}{r}\big)+\cl_3\big(\tfrac{3\pi}{r}\big)+\cdots+\cl_3\big(\tfrac{(2r-1)\pi}{r}\big)=-\tfrac3{4r^2}\ze(3),\label{equ-cl-fun}\\
&\cl_3\big(\tfrac{2\pi}{r}\big)+\cl_3\big(\tfrac{4\pi}{r}\big)+\cdots+\cl_3\big(\tfrac{(2r-2)\pi}{r}\big)=-\big(1-\tfrac1{r^2}\big)\ze(3).\label{equ-cl-fun-1}
\end{align}
Setting $r=4,5,6$ in \eqref{equ-cl-fun} and $r=5$ in \eqref{equ-cl-fun-1} one obtains
\begin{align}
&\cl_{3}\big(\tfrac{\pi}{4}\big)+\cl_{3}\big(\tfrac{3\pi}{4}\big)=-\tfrac3{128}\ze(3), \qquad
\cl_{3}\big(\tfrac{\pi}{5}\big)+\cl_{3}\big(\tfrac{3\pi}{5}\big)=\tfrac9{25}\ze(3), \label{equ-cl-fun-exa-1}\\
&\cl_{3}\big(\tfrac{\pi}{6}\big)+\cl_{3}\big(\tfrac{5\pi}{6}\big)=\tfrac1{12}\ze(3), \qquad
\cl_{3}\big(\tfrac{2\pi}{5}\big)+\cl_{3}\big(\tfrac{4\pi}{5}\big)=-\tfrac{12}{25}\ze(3),\label{equ-cl-fun-exa-2}
\end{align}
since $\cl_3(\theta)=\cl_3(2\pi\pm \theta)$, $\cl_3(\pi)=-\tfrac3{4}\ze(3)$ and $\cl_3\big(\tfrac{\pi}{2}\big)=-\tfrac3{32}\ze(3)$ by \eqref{equ:Riemmann}.

On the other hand, Davydychev and Kalmykov have shown that
(see \cite[(2.36), (2.37) and (2.67)]{DavydychevDe2004}) the Ap\'ery-type inverse binomial series
\begin{align*}
&\sum_{n=1}^\infty \frac{1}{n^3\binom{2n}{n}}u^n=2\cl_3(\theta)+2\theta \cl_2(\theta)-2\ze(3)+\theta^2\log\Big(2\sin\frac{\theta}{2}\Big), \\
&\sum_{n=1}^\infty \frac{H_{n-1}}{n^2\binom{2n}{n}}u^n=4\cl_3(\pi-\theta)-2\theta\cl_2(\pi-\theta)+3\ze(3), \\
&\sum_{n=1}^\infty \frac{H_{2n-1}}{n^2\binom{2n}{n}}u^n=-2\cl_3(\theta)+4\cl_3(\pi-\theta)-2\theta\cl_2(\pi-\theta)-\theta\cl_2(\theta)+5\ze(3),
\end{align*}
where $u=4\sin^2\frac{\theta}{2}$ for any $\theta\in [0,\pi]$.
From the above three equations we can find that
\begin{align}\label{equ-comb}
\sum_{n=1}^\infty \frac{u^n}{n^2\binom{2n}{n}}\left(H_{2n}-H_{n-1}\right)&=\sum_{n=1}^\infty \frac{u^n}{n^2\binom{2n}{n}}\left(H_{2n-1}-H_{n-1}+\frac1{2n}\right) \nonumber\\
&=\ze(3)-\cl_3(\theta)+\frac{\theta^2}{2}\log\Big(2\sin\frac{\theta}{2}\Big).
\end{align}
Taking special values of $u$ we then obtain
\begin{align*}
u=\Big(\frac{\sqrt{5}\pm 1}2\Big)^2=4\sin^2\frac{\theta}{2}\quad \Longrightarrow \quad& \sin\frac{\theta}{2}=\frac{\sqrt{5}\pm 1}{4}\quad \text{and}\quad \theta=\frac{2\pi}{5}\pm\frac{\pi}{5},\\
u=\frac{5\pm \sqrt{5}}{2}=4\sin^2\frac{\theta}{2}\quad \Longrightarrow \quad& \sin\frac{\theta}{2}=\sqrt{\frac{5\pm\sqrt{5}}{8}}\quad \text{and}\quad \theta=\frac{3\pi}{5}\pm\frac{\pi}{5},\\
u=2\pm\sqrt{3}=4\sin^2\frac{\theta}{2}\quad \Longrightarrow \quad&
\sin\frac{\theta}{2}=\frac{\sqrt{2\pm\sqrt{3}}}{2}\quad \text{and}\quad \theta=\frac{\pi}{2}\pm\frac{\pi}{3},\\
u=2\pm\sqrt{2}=4\sin^2\frac{\theta}{2}\quad \Longrightarrow \quad& \sin\frac{\theta}{2}=\frac{\sqrt{2\pm\sqrt{2}}}{2}\quad \text{and}\quad \theta=\frac{\pi}{2}\pm\frac{\pi}{4}.
\end{align*}
Therefore, \eqref{equ-cl-fun-exa-1}--\eqref{equ-comb} yield
\begin{align*}
\sum_{n=1}^\infty \frac{L_{2n}}{n^2\binom{2n}{n}}\left(H_{2n}-H_{n-1}\right)
&=2\ze(3)-\cl_{3}\big(\tfrac{3\pi}{5}\big)-\cl_{3}\big(\tfrac{\pi}{5}\big)+\frac{4\pi^2}{25}\log(\phi) \\
&=\frac{41\ze(3)+4\pi^2\log(\phi)}{25},\\
\sum_{n=1}^\infty \frac{v_n(5,5)}{n^2\binom{2n}{n}}\left(H_{2n}-H_{n-1}\right)
&=2\ze(3)-\cl_{3}\big(\tfrac{2\pi}{5}\big)-\cl_3\big(\tfrac{4\pi}{5}\big)+\frac{\pi^2\log\left(5^5\phi^6\right)}{50}\\
&=\frac{124\ze(3)+\pi^2\log\left(5^5\phi^6\right)}{50},\\
\sum_{n=1}^\infty \frac{v_n(4,1)}{n^2\binom{2n}{n}}\left(H_{2n}-H_{n-1}\right)
&=2\ze(3)-\cl_{3}\big(\tfrac{\pi}{6}\big)-\cl_{3}\big(\tfrac{5\pi}{6}\big)+\frac{\pi^2\log(2+\sqrt{3})}{6}\\
&=\frac{23\ze(3)+2\pi^2\log(2+\sqrt{3})}{12},\\
\sum_{n=1}^\infty \frac{v_n(4,2)}{n^2\binom{2n}{n}}\left(H_{2n}-H_{n-1}\right)
&=2\ze(3)-\cl_{3}\big(\tfrac{\pi}{4}\big)-\cl_{3}\big(\tfrac{3\pi}{4}\big)\\
& +\frac{\pi^2(5\log(2)+8\log(1+\sqrt{2}))}{64}\\
&=\frac{259\ze(3)+2\pi^2(5\log(2)+8\log(1+\sqrt{2}))}{128}.
\end{align*}

This concludes the proof of Theorem~\ref{thm-sun's-conj}. \hfill $\blacksquare$

\section{Ap\'ery-like sums via trigonometric integrals}

\subsection{Akhilesh's $n$-tails of Ap\'ery-type inverse binomial series}
In two sequential papers \cite{Akhilesh1, Akhilesh}, Akhilesh discovered profound connections between MZVs and the Ap\'ery-type inverse binomial series \eqref{defn:gs}, up to a slightly different normalization. His approach relies on studying the $n$-tails (and, more generally, double tails) of these series. By establishing an explicit formula for the depth-one case, the process can then be iterated to arbitrary depth. To make this framework more transparent, we reformulate one of his key results as follows. Under the convention $\binom{0}{0}=1$, we set
\begin{align*}
b_n(z):=\frac{4^n z^{2n}}{\binn} \quad \text{and}\quad b_n:=b_n(1)=\frac{4^n}{\binn} \quad \forall n\ge 0.
\end{align*}
Define
\begin{align}\label{defn-g}
g_s(t)=
\left\{
 \begin{array}{ll}
 \tan t\, dt, \quad & \hbox{if $s=1$;} \\
 dt \circ (\cot t\, dt)^{s-2}\circ dt, \quad & \hbox{if $s\ge 2$,}
 \end{array}
\right.
\end{align}
and their non-trigonometric counterpart
\begin{equation}\label{defn-G}
 G_{s}(t)=
 \left\{
\begin{array}{ll}
 \om_{2} & \hbox{if $s=1$;} \\
 \om_{1}\om_0^{s-2}\om_{1} \qquad & \hbox{if $s\ge 2$,}
 \end{array}
 \right.
\end{equation}
where
\begin{alignat}{5}\label{defn:om012}
\om_0:=&\, \frac{dt}{t}, \quad &\,\om_1:=&\,\frac{dt}{\sqrt{1-t^2}}, \quad &\,\om_2:=&\,\frac{t\,dt}{1-t^2}.
\end{alignat}

\begin{thm} \label{thm-gs-Akhilesh} \emph{(cf. \cite[Theorem~4]{Akhilesh})}
Let $n\in\N_0$, $\bfs=(s_1,\dots,s_d)\in\N^d$, $y\in(-\pi/2,\pi/2)$ if $s_1=1$ and $y\in[-\pi/2,\pi/2]$ if $s_1\ge 2$.
Then the $n$-tail of $\gs(\bfs;\sin y)$
\begin{align*}
\gs(\bfs;\sin y)_n:=&\, \sum_{n_1 > \cdots > n_d>n} \frac{b_{n_1}(\sin y)}{(2n_1)^{s_1}\cdots (2n_d)^{s_d}}\\
=&\, \frac{d}{dy} \int_0^y g_{s_1}\circ \cdots \circ g_{s_d}\circ b_{n}(\sin t) \,dt.
\end{align*}
Using non-trigonometric 1-forms, for all $z\in(-1,1)$ we have
\begin{align*}
\gs(\bfs;z)_n:=&\, \sum_{n_1 > \cdots > n_d>n} \frac{b_{n_1}(z)}{(2n_1)^{s_1}\cdots (2n_d)^{s_d}}\\
=&\, \sqrt{1-z^2}\frac{d}{dz} \int_0^z G_{s_1}\circ \cdots \circ G_{s_d}\circ b_{n}(t) \om_1.
\end{align*}
\end{thm}

We now provide a new proof of the following result conjectured by Sun \cite[(3.10)]{Sun2015} first proved
by Ablinger \cite[(5-56)]{Ablinger2015}.
\begin{thm} We have
\begin{equation}\label{equ:2conjSun}
\sum_{n\ge1}\frac{3^n}{n^2\binom{2n}{n}}\left(H_n+\frac1{2n}\right)=\frac{\pi^2}3 \log 3.
\end{equation}
\end{thm}
\begin{proof}
First, since $H_n=H_{n-1}+1/n$ we can rewrite \eqref{equ:2conjSun} as
\begin{align*}
8\gs(2,1;\sqrt{3}/2)+12\gs(3;\sqrt{3}/2)=\frac{\pi^2}3 \log 3.
\end{align*}
By Akhilesh's Theorem~\ref{thm-gs-Akhilesh} it suffices to show
\begin{equation}\label{equ:2conjSun2}
8A+12 B=\frac{\pi^2}3 \log 3,
\end{equation}
where
\begin{align*}
A=&\, \int_0^{\pi/3} dt (\tan t\, dt)\, dt,\quad
B=\int_0^{\pi/3} \frac{dt}{\tan t} \, dt\, dt=\frac12 \int_0^{\pi/3} (t^2 \cot t) dt.
\end{align*}
By integrating by parts, we get
\begin{align*}
A=\int_0^{\pi/3} \left(\int_0^t (x \tan x) dx\right) dt
=&\, \left. t\int_0^t (x \tan x) dx\right|_0^{\pi/3} - \int_0^{\pi/3} (t^2 \tan t) dt \\
=&\, \frac{\pi}{3} \int_0^{\pi/3} (x \tan x) dx - \int_0^{\pi/3} (t^2 \tan t) dt.
\end{align*}
We can verify by differentiation that (see also \cite[Lemma 1]{Akhilesh})
\begin{align*}
\int_0^z (t \tan t) dt=&\, \frac{i}2 \Big(z^2+\Li_2(-e^{2iz})+\frac12\zeta(2)\Big) - z\log(1 + e^{2iz}),\\
\int_0^z (t^2 \tan t) dt=&\, \frac{i}3 z^3 - z^2 \log(1 + e^{2iz}) + iz \Li_2(-e^{2iz}) - \frac12 \Li_3(-e^{2iz}) -\frac38\zeta(3), \\
\int_0^z (t^2 \cot t) dt=&\, -\frac{i}3 z^3 + z^2\log(1 + e^{iz})+ z^2\log(1 - e^{iz}) - 2i z \Li_2(e^{iz}) \\
&\, - 2iz \Li_2(-e^{iz})+ 2 \Li_3(e^{iz}) + 2 \Li_3(-e^{iz})-\frac12\zeta(3),
\end{align*}
where we have used the fact that $\Li_2(-1)=\zeta(\bar2)=-\frac12\zeta(2)=-\frac{\pi^2}{12}$ and $\Li_3(-1)=\zeta(\bar3)=-\frac34\zeta(3)$.
Setting $\eta=e^{2\pi i/3}$, after a lot of simplification and cancelations we obtain
\begin{align}\label{equ:showIm=0}
8A+12B= \frac{\pi^2 \log 3}{3} -\frac{2\pi i}3 \left(\frac{\pi^2}9 + 2 \Li_2(-\eta) + 3\Li_2(\eta)\right).
\end{align}
Noting that by the series definition
\begin{equation}\label{equ:sqForm}
\Li_2(-\eta) + \Li_2(\eta) =\sum_{k=1}^\infty \frac{(-\eta)^{k}+\eta^k}{k^2}=\frac12 \Li_2(\eta^2),
\end{equation}
we know the quantity in the parenthesis in \eqref{equ:showIm=0}
\begin{align}\label{equ:Its0}
\frac{\pi^2}{9}+2 \Li_2(-\eta) + 3\Li_2(\eta)=\frac{\pi^2}{9}+\Li_2(\eta^2)+\Li_2(\eta)
\end{align}
is real (invariant under conjugation) so that it must be zero
as $8A+12B$ is real. Thus \eqref{equ:2conjSun} follows.

We can also prove directly \eqref{equ:Its0} vanishes as follows. By \eqref{equ:sqForm} we have
\begin{align*}
2 \Li_2(-\eta) + 3\Li_2(\eta)+\Li_2(1) =&\, \Li_2(\eta^2)+\Li_2(\eta)+\Li_2(1) \\
=&\, \sum_{k=1}^\infty \frac{\eta^{2k}+\eta^k+1}{k^2}=\frac13 \Li_2(1)
\end{align*}
since $\eta^{2k}+\eta^k+1=0$ if $3\nmid k$ and $\eta^{2k}+\eta^k+1=3$ if $3\mid k$.
This proves \eqref{equ:Its0} since $\Li_2(1) =\zeta(2)=\pi^2/6$. We have completed the proof of the theorem.
\end{proof}

\subsection{Odd-indexed variations of Ap\'ery-type inverse binomial series}
By extending Akhilesh's idea it is not too difficult to discover the following recursive formulas. Define
\begin{align*}
f_1(t):=1,\quad f_2(t):= \frac{t}{\sqrt{1-t^2}},\quad f_3(t)=\frac1t, \quad f_{20}(t):= \frac{1}{t(1\mp t^2)}, \quad f_5(t)=t.
\end{align*}
Then by applying $t\to \sin^{-1} t$ in Theorem~\ref{thm-gs-Akhilesh} and its odd-indexed variations we find the following
iterative formulas.

\begin{thm} \label{thm-All-bn-it}
For any positive integer $n$ and $|z|<1$ we have
\begin{align}
 \sum_{m>n} \frac{b_{m}(z)}{2m} = &\, f_2(z)\int_0^z b_{n}(t) \om_1, \label{bn-it1} \\
 \sum_{m>n} \frac{b_{m}(z)}{(2m)^s}=&\, f_1(z) \int_0^z \om_0^{s-2} \om_1\, b_{n}(t)\om_1\quad\forall s\ge 2,\label{bn-it2} \\
 \sum_{m\ge n} \frac{b_{m}(z)}{2m+1}=&\,f_{20}(z)\int_0^z b_{n}(t) \om_1, \label{bn-it3} \\
 \sum_{m\ge n} \frac{b_{m}(z)}{(2m+1)^s}=&\,f_3(z)\int_0^z \om_0^{s-2} \om_3\, b_{n}(t) \om_1 \quad\forall s\ge 2, \label{bn-it4} \\
 \sum_{m>n} \frac{b_{m}(z)}{2m-1} = &\, f_5(z) \int_0^z \om_3\, b_{n}(t) \om_1 +f_2(z) \int_0^z b_{n}(t) \om_1, \label{bn-it5} \\
 \sum_{m>n} \frac{b_{m}(z)}{(2m-1)^{s}} = &\, f_5(z) \int_0^z (\om_0+1) \om_0^{s-2} \om_3\, b_{n}(t) \om_1\quad\forall s\ge 2, \label{bn-it6}
\end{align}
where $\om_j$ ($0\le j\le 2$) are defined by \eqref{defn:om012},
\begin{align*}
\om_3=\frac{dt}{t\sqrt{1-t^2}} \quad\text{and}\quad \om_{20}=\om_2+\om_0=\frac{dt}{t(1-t^2)}
\end{align*}
such that $f_j(t)\om_1=\om_j$ for all $j\ge 1$.
\end{thm}

\begin{proof}
We outline only the main ideas of the proof of \eqref{bn-it1}--\eqref{bn-it6}. Let $u_n(z)$ (resp.\ $v_n(z)$) denote the left-hand (resp.\ right-hand) side of any of the equations listed above. By straightforward differentiation, one can show that $u_n(z) - u_{n+1}(z) = v_n(z) - v_{n+1}(z)$. Because $\lim_{n\to\infty} u_n(z) = \lim_{n\to\infty} v_n(z) = 0$, the desired identities follow immediately by telescoping. The remaining verification is a routine application of calculus and is omitted here.
\end{proof}

\begin{ex}
As an easy example of Theorem \ref{thm-All-bn-it} we apply \eqref{bn-it2} followed by \eqref{bn-it3} to get
\begin{align*}
 \sum_{n_1>n_2\ge 0} \frac{ b_{n_1}}{n_1^2 (2n_2+1)}
=&\, 4\int_0^1 \om_1 \circ \om_{20} \circ\om_1.
\end{align*}
By the change of variables $t\to (1-t^2)/(1+t^2)$ we  have
\begin{align*}
 \sum_{n_1>n_2\ge 0} \frac{ b_{n_1}}{n_1^2 (2n_2+1)}
=&\,- 4\int_0^1 (\tx_{-i}-\tx_{i})\circ (\tx_{0}+\tx_{-1}+\tx_{1})\circ (\tx_{-i}-\tx_{i})=7\ze(3),
\end{align*}
where we have used Au's Mathematica package \cite{Au2020} at the last step.
\end{ex}

\subsection{Ap\'ery-type binomial series}
Akhilesh's iterative approach for evaluating Ap\'ery-type inverse binomial series can be generalized to standard binomial series, as well as to configurations where the summation indices possess arbitrary parity. To achieve this, we must extend the classical definition of iterated integrals. For any 1-forms $f_1(t)\,dt,f_2(t)\,dt$ and functions $F(t)$ and $G(t)$, we define
\begin{align*}
 \big (F(t)+ f_1(t)\,dt\circ G(t) \big)\circ f_2(t)\,dt := F(t)f_2(t) \,dt+ f_1(t)\,dt \circ G(t)f_2(t)\,dt.
\end{align*}

Put
\begin{equation*}
a_0(x):=1, \quad a_n(x):=\binom{2n}{n}\frac{x^{2n}}{4^n}  \quad\text{and} \quad a_n=a_n(1) \quad \forall n\ge 0.
\end{equation*}
Define
\begin{align}
p_s(t)=&\, \tan t\,dt\,(\cot t \,dt)^{s-1} ( 1-\csc t \,dt\circ\sec t). \label{defn-p}
\end{align}

\begin{thm} \label{thm-ga}
For all $n\in\N_0$, $\bfs=(s_1,\dots,s_d)\in\N^d$ and $y\in[-\pi/2,\pi/2]$ we have
\begin{align}\label{defn-ga}
\ga(\bfs;\sin y):=&\, \sum_{n_1>\cdots > n_d>n} \frac{a_{n_1}(\sin y)}{(2n_1)^{s_1}\cdots (2n_d)^{s_d}}\nonumber \\
=&\, \cot y \frac{d}{dy} \int_0^y p_{s_1} \circ\cdots \circ p_{s_d} \circ a_{n}(\sin t) \tan t\,dt.
\end{align}
Here if $y=\pm \pi/2$ then the last expression is defined to be the limit as $y\to \pm \pi/2$.
\end{thm}

When computing sums with both even and odd summation indices we may use the following iterative formulas.
Let $y\in(-\pi/2,\pi/2)$ and $F_s= (\cot t\, dt)^{s-1}$ for $s\in\N$. Then
\begin{align}
\sum_{m>n} \frac{a_m(\sin y)}{(2m)^s} =&\, \int_0^y F_s (1-\csc t \,dt\circ\sec t)a_n(\sin t) \tan t\,dt,\label{an-iterate1}\\
\sum_{m>n} \frac{a_m(\sin y)}{(2m+1)^s} =&\,\csc y \int_0^y F_s (1-dt\circ\csc t\sec t)a_n(\sin t) \sin t\tan t\,dt. \label{an-iterate2}\\
\sum_{m>n} \frac{a_m(\sin y)}{2m-1} =&\,\cos y\int_0^y a_n(\sin t) \tan t \sec t \, dt, \label{an-iterate3}\\
\sum_{m>n} \frac{a_m(\sin y)}{(2m-1)^s}
=&\, \sin y \int_0^y F_{s-1} \frac{dt}{\tan^2 t} a_n(\sin t) \tan t \sec t \, dt \quad (s\ge2).\label{an-iterate4}
\end{align}

All the formulas \eqref{an-iterate1}-- \eqref{an-iterate4} as well as Theorem~\ref{thm-ga} can be proved
using exactly the same idea as that of Theorem~\ref{thm-All-bn-it}. We thus omit their proofs here.

\begin{ex} \label{eg:illum}
We now compute the following Ap\'ery-like sum as an illuminating example using the above formulas:
\begin{equation*}
\sum_{n>m> 0} \frac{a_{n}^2}{(2n-1)m^2}=\sum_{n>m> 0} \binn^2 \frac{1}{16^n (2n-1)m^2}.
\end{equation*}
First, by \eqref{an-iterate3} and then \eqref{an-iterate1} we get
\begin{align}
& \sum_{n>m> 0} \frac{a_{n} \sin^{2n} y}{(2n-1)(2m)^2}
= \cos y\int_0^y \sum_{m>0} \frac{a_m(\sin t)}{(2m)^2} \tan t \sec t \, dt \nonumber\\
=&\, \cos y\int_0^y \tan t \sec t \, dt\, \cot t\, dt\, (1-\csc t \,dt\circ\sec t) \tan t\,dt\nonumber\\
=&\, \cos y\int_0^y \tan t \sec t \, dt\, \cot t\, dt\, (\tan t\,dt-\csc t \,dt\circ\sec t\tan t\,dt)\nonumber\\
=&\, \cos y\int_0^y \tan t \sec t \, dt\, \cot t\, dt\, \Big(\csc t-\cot t\Big) \,dt. \label{equ:sqEx1}
\end{align}
Note that by the beta integral
\begin{align}\label{equ:betaIntg}
\int_0^{\pi/2} \sin^{2n} t \,dt= \frac12 B \Big(n+\frac12,\frac12\Big)=\frac{\pi}{2}a_{n}.
\end{align}
Thus integrating \eqref{equ:sqEx1} over $(0,\pi/2)$ and dividing by $\pi/2$ we get
\begin{align*}
 &\, \sum_{n>m> 0} \frac{a_{n}^2}{(2n-1)m^2}
=\frac{8}{\pi} \int_0^{\pi/2} \cos t\,dt\, \,d(\sec t) \cot t \,dt \Big(\csc t-\cot t\Big) \,dt \nonumber \\
=&\, \frac{8}{\pi} \int_0^{\pi/2} dt\, \cot t \,dt \Big(\csc t-\cot t\Big) \,dt
-(1-\sin t) \csc t \,dt \Big(\csc t-\cot t\Big)\,dt.
\end{align*}
For any $N$th roots of unity $\xi$ and $\xi'$ we set $\tx_\xi=dt/(\xi-t)$ and $\td_{\xi,\xi'}=\tx_\xi-\tx_{\xi'}$.
Then under the change of variables $t\to \sin^{-1} [(1-t^2)/(1+t^2)]$ we have
\begin{align}\label{equ-1forms-ta-tx}
dt \to i \td_{-i,i}, \
\cot t\,dt \to \ty,\
%\tan t\,dt \to \tz,\
\csc t \,dt \to \td_{-1,1}, \
%\sec t \,dt \to -\ta, \
%\sec t\csc t \,dt \to \ty+\tz,
\end{align}
where $\ty=\tx_{-i}+\tx_{i}-\tx_{-1}-\tx_{1}$ and $\tz=-\ta-\tx_{-i}-\tx_{i}$.
Hence
\begin{align*}
\sum_{n>m> 0} \frac{a_{n}^2}{(2n-1)m^2}
=&\, \frac{8 i}{\pi} \int_0^1 \left\{
\aligned
 & (2\tx_{-1}-\tx_{-i}-\tx_{i})(1-\ty) \td_{-i,i} \\
 & \ \quad -(2\tx_{-1}-\tx_{-i}-\tx_{i})\td_{-1,1}
\endaligned \right\}.
\end{align*}
Now we may convert all the iterated integral terms into CMZVs of level 4 and then use
Au's Mathmatica package \cite{Au2020} to reduce the above to
\begin{align*}
\sum_{n>m> 0} \frac{a_{n}^2}{(2n-1)m^2}
=&\, \frac{3\pi^2}{2} -\frac{\pi}{3} -6\log^2 2+8\log 2-
 \frac{64}{\pi}{\rm Im}\Li_3\Big(\frac{1+i}2\Big)-\frac{16G}{\pi}, % \approx 0.03529309.
\end{align*}
where Im$\,z$ is the imaginary part of $z$ and
$G=\sum_{k\ge 0}\frac{(-1)^k}{(2k+1)^2}$ is Catalan's constant.
One can check that numerically the sum is approximately 0.14117237.
\end{ex}

By utilizing the iterations \eqref{an-iterate1}--\eqref{an-iterate4} and applying the same change of variables $t \mapsto \sin^{-1} [(1-t^2)/(1+t^2)]$, one can show that the Ap\'ery-like sums $\gs(\bfs;1)$ defined by \eqref{defn:gs}, $\ga(\bfs;1)$ defined by \eqref{defn-ga}, and their odd-indexed variations can all be expressed as $\Q$-linear combinations of the real or imaginary parts of level-4 CMZVs. Furthermore, by invoking the beta integral \eqref{equ:betaIntg}, analogous results are obtained when $a_n$ and $b_n$ are replaced by $a_n^2$ and $b_n^2$, respectively. The only notable difference is that an additional factor of $1/\pi$ is occasionally required, as illustrated by Example~\ref{eg:illum}. For more comprehensive treatment, we refer the reader to \cite{XuZhao2022,XuZhao2025Apr,XuZhao2026Mar}.

\subsection{Alternating Ap\'ery-like sums}
One such sum in \eqref{equ:AperySeries} already appeared in Ap\'ery's original work. Generally, for any $n\in\N$, $\bfeta=(\eta_1,\dots,\eta_d)\in\{\pm 1\}^d$, $\bfs=(s_1,\dots,s_d)\in\N^d$ and a complex variable $z$,
we define
\begin{align*}
\ga(\bfs;\bfeta;z)_n:=\sum_{n_1>\dots>n_d>n}
 \frac{a_{n_1}(z)\eta_1^{n_1}\cdots \eta_d^{n_d}}{(2n_1)^{s_1}\cdots (2n_d)^{s_d}},\\
\gs(\bfs;\bfeta;z)_n:=\sum_{n_1>\dots>n_d>n}
 \frac{b_{n_1}(z)\eta_1^{n_1}\cdots \eta_d^{n_d}}{(2n_1)^{s_1}\cdots (2n_d)^{s_d}},
\end{align*}
which are called the $n$-\emph{tails} of the alternating Ap\'ery-type series
\begin{align*}
\ga(\bfs;\bfeta;z):=\ga(\bfs;\bfeta;z)_0 \quad\text{and}\quad
\gs(\bfs;\bfeta;z):=\gs(\bfs;\bfeta;z)_0,
\end{align*}
respectively.
To save space we will put a bar on top of $s_j$ if the corresponding $\eta_j=-1$. For example,
\begin{align*}
 \gs(\bar3,4,\bar1;z)\,&=\sum_{n_1>n_2>n_3>0}
 \frac{b_{n_1}(z) (-1)^{n_1+n_3}}{(2n_1)^{3}(2n_2)^{4} (2n_3) } .
\end{align*}

By modifying the results in the previous sections we can derive the corresponding formulas for alternating sums.
The key is to replace the trigonometric 1-forms by the hyperbolic ones. To fix notation, we put
\begin{alignat*}{4}
 \sh z &\,= -i \sin (i z)&\,=\frac{e^z - e^{-z}} {2},&&\quad
 \ch z &\,= \cos (i z)&\,=\frac{e^z + e^{-z}} {2},\\
 \th z &\,= -i \tan(i z)&\,=\frac{e^z - e^{-z}}{e^z + e^{-z}},&&\quad
 \cth z &\,= i \cot(i z)&\,=\frac{e^z + e^{-z}} {e^z - e^{-z}}
%,\\
% \sech z &\,= \ \ \sec (i z)&\,=\frac{2} {e^z + e^{-z}},&&\quad
% \csch z &\,= i \csc (i z)&\,=\frac{2} {e^z - e^{-z}}
.
\end{alignat*}
Define the hyperbolic counterpart of $p_s(t)$ and $g_s(t)$ by
\begin{align*}
\tlp_s(t)=&\, \th t\,dt\, (\cth t\,dt)^{s-1} ( 1-\csch t \,dt\circ\sech t), \\
\tlg_s(t)=&\,
\left\{
 \begin{array}{ll}
 \th t\, dt, \quad & \hbox{if $s=1$;} \\
 dt \circ (\cth t\, dt)^{s-2}\circ dt, \quad & \hbox{if $s\ge 2$.}
 \end{array}
\right.
\end{align*}

\begin{thm} \label{thm-gsbar1-ItInt}
Set $\psi=\sh^{-1} 1=\log(1+\sqrt{2})$. For all $n\in\N_0$, $\bfs=(s_1,\dots,s_d)\in\N^d$ and $y\in[-\psi,\psi]$, we have
\begin{align*}
\ga(\bfs;i\sh y)_n=&\, (-1)^d \cth y \frac{d}{dy} \int_0^y \tlp_{s_1}\circ\cdots\circ\tlp_{s_d}\circ a_{n}(i\sh t) \th t\,dt,\\
\gs(\bfs;i\sh y)_n=&\,(-1)^d \frac{d}{dy}\int_0^y \tlg_{s_1}\circ \cdots \circ \tlg_{s_d}\circ b_n(i\sh t) \,dt.
\end{align*}
\end{thm}

\begin{proof}
By the substitution $y\to iy$ in Theorem~\ref{thm-gs-Akhilesh}, we obtain immediately
\begin{align*}
\gs(\bfs;i\sh y)_n:=&\, \sum_{n_1 > \cdots > n_d>n}
\frac{b_{n_1}(i\sh y)}{(2n_1)^{s_1}\cdots (2n_d)^{s_d}} \\
=&\, -i\frac{d}{dy} \int_0^y g_{s_1}(it)\circ \cdots \circ g_{s_d}(it)\circ b_n(i\sh t) \,d(it),
\end{align*}
where we have used the fact that $\sin it=i\sh t$ and $\cos it=\ch t$. Note that
\begin{align*}
g_s(i t)= &
\left\{
 \begin{array}{ll}
 \tan (it)\, d(it), \quad & \hbox{if $s=1$;} \\
 d(it) \circ (\cot (it)\, d(it))^{s-2}\circ d(it), \ \, \quad\quad & \hbox{if $s\ge 2$}
 \end{array}
\right. \\
= &
\left\{
 \begin{array}{ll}
 -\th t\, dt, \quad & \hbox{if $s=1$;} \\
 -dt \circ (\cth t\, dt)^{s-2}\circ dt, \qquad\qquad \qquad & \hbox{if $s\ge 2$.}
 \end{array}
\right.
\end{align*}
The claim for $\gs(\bfs;i\sh y)_n$ follows immediately. The proof of the claim for $\ga(\bfs;i\sh y)_n$ is completely similar and is omitted here.
\end{proof}

By combining Theorems~\ref{thm-gs-Akhilesh} and \ref{thm-gsbar1-ItInt}, we can express all alternating Ap\'ery-type inverse binomial series in terms of iterated integrals. Moreover, if $(s_1, \eta_1) \neq (1, 1)$, then $\gs(\bfs;\bfeta;1)$ can be expressed as a $\Q[i, \sqrt{2}]$-linear combination of level-8 CMZVs. To extend this result to all odd-indexed variations, we first consider the alternating analogues of \eqref{bn-it1}--\eqref{bn-it6}. Write
\begin{align*}
b^+_n(z):=\frac{4^nz^{2n}}{\binn}, \quad b^-_n(z)=(-1)^nb_n(z),
 \quad \text{and}\quad b^-_n=b^-_n(1)\quad \forall n\ge 0.
\end{align*}
Define
\begin{align*}
f_{\pm 1}(t) :=1,\ &\, f_{\pm 2}(t):= \frac{t}{\sqrt{1\mp t^2}},\
f_{\pm 3}(t)=\frac1t, \ f_{\pm 20}(t):= \frac{1}{t(1\mp t^2)}, \ f_{\pm 5}(t)=t,\\
&\,\om_{\pm 1}:=\frac{dt}{\sqrt{1\mp t^2}},\qquad
%\om_{\pm 4}=\frac{t\, dt}{\sqrt{1-t^4}}\quad
\om_{\pm j}:=f_{\pm j}(t)\cdot \om_{\pm 1} \quad \forall j\ge 2.
\end{align*}
Then one can show that $\om_{\pm 20}=\om_0\mp \om_{\pm2}$ and by \eqref{bn-it1}--\eqref{bn-it6}
\begin{align}
\sum_{n_1>n} \frac{ b^\pm_{n_1}(z)}{2n_1} = &\,\pm f_{\pm 2}(z)\int_0^z b^\pm_{n}(t) \om_{\pm 1}, \label{bn-ALT-it1} \\
 \sum_{n_1>n} \frac{ b^\pm_{n_1}(z)}{(2n_1)^s}=&\, \pm f_1(z) \int_0^z \om_0^{s-2} \om_{\pm 1}\, b^\pm_{n}(t)\om_{\pm 1}\quad\forall s\ge 2,\label{bn-ALT-it2} \\
 \sum_{n_1\ge n} \frac{ b^\pm_{n_1}(z)}{2n_1+1}=&\,f_{\pm20}(z)\int_0^z b^\pm_{n}(t) \om_{\pm1} , \label{bn-ALT-it3} \\
 \sum_{n_1\ge n} \frac{ b^\pm_{n_1}(z)}{(2n_1+1)^s}=&\,f_3(z)\int_0^z \om_0^{s-2} \om_{\pm3}\, b^\pm_{n}(t) \om_{\pm1} \quad\forall s\ge 2, \label{bn-ALT-it4} \\
 \sum_{n_1>n} \frac{ b^\pm_{n_1}(z)}{2n_1-1} = &\, \pm f_5(z) \int_0^z \om_{\pm3}\, b^\pm_{n}(t) \om_{\pm1}\pm f_{\pm2}(z) \int_0^z b^\pm_{n}(t) \om_{\pm1}, \label{bn-ALT-it5} \\
 \sum_{n_1>n} \frac{ b^\pm_{n_1}(z)}{(2n_1-1)^{s}} = &\,\pm f_5(z) \int_0^z (\om_0+1) \om_0^{s-2} \om_{\pm3}\, b^\pm_{n}(t) \om_{\pm1}\quad\forall s\ge 2. \label{bn-ALT-it6}
\end{align}

We illustrate this with an example. For the proof of the above formulas in general we refer the interested reader to \cite{XuZhao2022}.
\begin{ex}
In depth two, using \eqref{bn-ALT-it1} and \eqref{bn-ALT-it3} successively we get
\begin{align*}
S:=&\, \sum_{n>m\ge 0}\frac{(-1)^n b_{n}}{(2n)(2m+1)}
=\frac{-1}{\sqrt{2}} \int_0^1 \sum_{m\ge 0} \frac{b^{-}_m(t)}{2m+1} \om_{-1}
=\frac{-1}{\sqrt{2}} \int_0^1 \om_{-20} \om_{-1}.
\end{align*}
By the change of variables $t\to \sqrt{2}t/\sqrt{1+t^4}$ we see that
\begin{alignat*}{6}
\om_0:=&\frac{dt}{t} & \, \quad\to & \quad \quad \frac{(1-t^4)\, dt}{t(1+t^4)}
 &=&\, \ta+\frac12 \sum_{j=1}^4 \tx_{\mu_j}, \\ %\label{rho:ChangVar0}\\
\om_{-1}:=&\frac{dt}{\sqrt{1+t^2}} & \, \quad\to & \quad \frac{\sqrt{2} (1-t^2) dt}{1+t^4}
 &=&\, \frac{\sqrt{2}}{4}\sum_{j=1}^4 (\mu_j-\mu_j^3)\tx_{\mu_j}, \\ % \label{rho:ChangVar-1}\\
\om_{-2}:=&\frac{t\,dt}{1+t^2} & \, \quad\to & \quad \frac{2 (1-t^2) dt}{(1+t^2)(1+t^4)}
 &=&\, \frac12\sum_{j=1}^4 \tx_{\mu_j}-\tx_i-\tx_{-i}, % \label{rho:ChangVar-2}
\end{alignat*}
where $\mu_j=\exp((2j-1)\pi i/4)$ ($1\le j\le 4$) are the four 8th roots of unity satisfying $x^4+1=0$. Thus
\begin{align*}
S=&\, -\frac{1}{4}\int_0^1 \left(\ta+\tx_i+\tx_{-i}\right)
 \left( \sum_{j=1}^4 (\mu_j-\mu_j^3)\tx_{\mu_j} \right)\\
= &\,-\frac{1}{4} \sum_{j=1}^4 (\mu_j-\mu_j^3) \Li_{2}(\mu_j^{-1})
-\frac{1}{4}\sum_{j=1}^4\sum_{\eps=\pm i} (\mu_j-\mu_j^3) \Li_{1,1}(\eps^{-1},\eps/\mu_j)\\
=&\, \frac{\sqrt{2}}4 \Big(\log 2 \log\nu - \log^2\nu - 2 \Li_2(\nu^{-1})\Big)-\frac{\sqrt{2}\pi^2}{96}
%-\[Pi]^2 + 24 (Log[2] Log[1 + Sqrt[2]] - Log[1 + Sqrt[2]]^2 - 2 PolyLog[2, -1 + Sqrt[2]])
 \approx -0.53464318757
\end{align*}
where $\nu=1 + \sqrt{2}$ by Au's Mathematica package \cite{Au2020}.
\end{ex}

Now we consider the alternating form of \eqref{an-iterate1}--\eqref{an-iterate4}. By the change of variables $y\to \sh^{-1}(iz)$
we can easily get
\begin{align}
 \sum_{m>n} \frac{a^-_m(z)}{(2m)^s} =&\, \int_0^z \om_0^{s-1} \Big(\om_{-3}\circ \frac{1}{\sqrt{1+t^2}}-1\Big)a^-_n(t) \om_{-2},\label{an-ALT-it1}\\
\sum_{m\ge n} \frac{a^-_m(z)}{(2m+1)^s} =&\,\frac1z\int_0^z \om_0^{s-1} \Big(\frac1t+\om_{-1}\circ\frac{1}{\sqrt{1+t^2}}\Big) a^-_n(t) \om_{-2},\label{an-ALT-it3}\\
\sum_{m>n} \frac{a^-_m(z)}{2m-1} =&\,-\sqrt{1+z^2}\int_0^z a^-_n(t) \frac{t\, dt}{(1+t^2)^{3/2}}, \label{an-ALT-it4}\\
\sum_{m>n} \frac{a^-_m(z)}{(2m-1)^s}
=&\,z \int_0^z \om_0^{s-2} \frac{\sqrt{1+t^2}\, dt}{t^2} a^-_n(t) \frac{t\, dt}{(1+t^2)^{3/2}} \quad (s\ge2).\label{an-ALT-it5}
\end{align}

\begin{ex}
Using \eqref{an-ALT-it3} and \eqref{an-ALT-it1}, we get
\begin{align*}
 \sum_{n\ge m>0}\frac{(-1)^n a_n}{(2n+1)(2m)}
=&\, \int_0^1 \Big(\frac1t-\om_{-1}\circ \frac{1}{\sqrt{1+t^2}}\Big) \sum_{m>0}\frac{a^-_m(t)}{2m} \om_{-2}\\
=&\, \int_0^1 \Big(\frac1t-\om_{-1}\circ \frac{1}{\sqrt{1+t^2}}\Big)\om_{-2}\Big(\om_{-3}\circ \frac{1}{\sqrt{1+t^2}}-1\Big) \om_{-2}\\
=&\, \int_0^1 \Big(\frac{dt}{1+t^2}-\om_{-1}\circ\frac{t\, dt}{(1+t^2)^{3/2}}\Big) \Big(\om_{-3}-\om_{-20}-\om_{-2}\Big) \\
=&\, \frac{1}{\sqrt{2}} \int_0^1 \om_{-1} \big(\om_{-3}-\om_{-20}-\om_{-2}\big).
\end{align*}
Set $\mu=\exp(2\pi i/8)$. By the change of variables $t\to i(1-t^2)/(1+t^2)$ we get
\begin{align*}
&\, \sum_{n\ge m>0}\frac{(-1)^n a_n}{(2n+1)(2m)}
= \frac{1}{\sqrt{2}} \int_{1}^{\mu} \td_{i,-i}\big(2\tx_{-1}-\tx_{-i}-\tx_{i}\big) \\
=&\, \frac{1}{\sqrt{2}} \left(\int_{1}^0 \td_{i,-i}\big(2\tx_{-1}-\tx_{-i}-\tx_{i}\big)
+\int_{0}^{\mu} \td_{i,-i}\int_{1}^0\big(2\tx_{-1}-\tx_{-i}-\tx_{i}\big)  \right.\\
&\, \ \hskip5cm \left. +\int_{0}^{\mu}\td_{i,-i}\big(2\tx_{-1}-\tx_{-i}-\tx_{i}\big) \right)\\
=&\, \frac{1}{\sqrt{2}} \left(\int_0^1 \big(2\tx_{-1}-\tx_{-i}-\tx_{i}\big) \td_{i,-i}
-\int_{0}^{1}\td_{i/\mu,-i/\mu} \int_0^1\big(2\tx_{-1}-\tx_{-i}-\tx_{i}\big) \right.\\
&\, \ \hskip5cm \left. +\int_{0}^{1} \td_{i/\mu,-i/\mu}\big(2\tx_{-1/\mu}-\tx_{-i/\mu}-\tx_{i/\mu}\big) \right) \\
=&\, \frac{\pi^2}{12}+ 2\log2\log(1 + \sqrt2) - \frac32 \log^2(1 + \sqrt2) - 2\Li_2(\sqrt2-1)
\approx -0.05598682674
% \[Pi]^2/12 + Log[4] Log[1 + Sqrt[2]] - 3/2 Log[1 + Sqrt[2]]^2 - 2 PolyLog[2, -1 + Sqrt[2]]
% 2*ColoredMZV[8, {1, 1}, {4, 2}] - 2*ColoredMZV[8, {1, 1}, {4, 6}]+ ColoredMZV[8, {1, 1}, {6, 4}] + ColoredMZV[8, {1, 1}, {2, 0}]- ColoredMZV[8, {1, 1}, {2, 4}] - ColoredMZV[8, {1, 1}, {6, 0}]+(-ColoredMZV[8, {1}, {6}] - ColoredMZV[8, {1}, {2}] + 2*ColoredMZV[8, {1}, {4}])*(-ColoredMZV[8, {1}, {7}] + ColoredMZV[8, {1}, {3}])+2*ColoredMZV[8, {1,1}, {7, 6}] - 2*ColoredMZV[8, {1,1}, {3, 2}]+ ColoredMZV[8, {1,1}, {3, 0}] + ColoredMZV[8, {1,1}, {3, 4}]- ColoredMZV[8, {1,1}, {7, 4}] - ColoredMZV[8, {1,1}, {7, 0}]
%
\end{align*}
by using Au's Mathematica package \cite{Au2020}.
\end{ex}

We will conclude this section by extending our questions on page \pageref{question}:

\medskip\noindent
{\bf Question 3.}
Let $d\in \N$, $\bfs=(s_1,\dots,s_d)\in\N^d$ and $\bfeta=(\eta_1,\dots,\eta_d)\in \{\pm1\}^d$.
Let $l_j(n)=2n$ or $l_j(n)=2n\pm 1$ for all $1\le j\le d$. Suppose $z=\sqrt{j}/2$ for $j=1,2,4$. Is there a level $N$ such that
\begin{align*}
 \sum_{n_1 \succ n_2 \succ\, \cdots \succ n_d\succ\,0}
 \frac{a_{n_1}z^{2n_1}\eta_1^{n_1}\cdots \eta_d^{n_d}}{l_1(n_1)^{s_1}\cdots l_d(n_d)^{s_d}}\in\CMZV^N\otimes\Q[\exp(2\pi i/N)],\\
 \sum_{n_1 \succ n_2 \succ\, \cdots \succ n_d\succ\,0}
 \frac{b_{n_1}z^{2n_1}\eta_1^{n_1}\cdots \eta_d^{n_d}}{l_1(n_1)^{s_1}\cdots l_d(n_d)^{s_d}}\in\CMZV^N\otimes\Q[\exp(2\pi i/N)]
\end{align*}
if the sums converge? Here ``$\succ$'' can be either ``$\ge$'' or ``$>$''. One may also ask the same question when
$a_n$ and $b_n$ are replaced by their squares while adding an extra factor of $1/\pi$ on the
right-hand side for $a_n^2$ version. Available evidence seems to suggest that $N$ is at least 120.

\section{Ap\'ery-like sums via fibrations of iterated integrals}
In \cite{CampbellGZ2024}, Campbell et al.~solved an open problem posed by H. Chen in his book \cite{HJChen2013} by evaluating the following two Ap\'ery-like sums explicitly and showed that
\begin{equation*}
 \sum_{n=0}^\infty \frac{1}{(2n+1)^3\binom{2n}{n}}\in i \CMZV_3^{12}, \qquad
 \sum_{n=0}^\infty \frac{(-1)^n}{(2n+1)^3\binom{2n}{n}}\in  \CMZV_3^{10}.
\end{equation*}
In fact, they were able to show that
\begin{equation*}
\sum_{n=0}^\infty \frac{4^n\sin^{2n+1}(2m\pi/N)}{(2n+1)^k\binom{2n}{n}}\in i \CMZV_k^{\rm{lcm}(2,N)}\quad \forall k\ge 2,N\ge 3,m\in \Z.
\end{equation*}
The main idea to evaluate these is to use a variation of \eqref{equ:fn} in the form of
\begin{align*}
 \int_{0}^{1} \bigg( \frac{x}{1 + x^2} \bigg)^{2n+1} \frac{d x}{1+x^2}
 = \frac{1}{ 4 (2 n+1) \binom{2 n}{n}}
\end{align*}
together with the identity
$$
\sum_{n=0}^{\infty} \frac{z^{2n+1}}{(2n+1)^k} = \frac{\Li_k(z)-\Li_k(-z)}{2},
$$
to derive (for any $k\ge 3$ and $w$ such that $ \Im(i w)>0$ and $ \Im(i /w)>0$)
\begin{align*}
&\sum _{n=0}^{\infty } \frac{(-1)^n }{(2 n+1)^k \binom{2 n}{n}}\biggl(\frac{1-w^2}{w}\biggr)^{2 n+1}\\
=&{}2i \int_0^1 \frac{ \Li_{k-1}\bigl(\frac{1-w^2}{i w}\frac{x}{1+x^2}\bigr)-\Li_{k-1}\bigl(-\frac{1-w^2}{i w}\frac{x}{1+x^2}\bigr)}{1+x^2} \,dx\\
=&{}\frac{2i }{(k-2)!}\int_{i w}^i \frac{\log^{k-2}\bigl( \frac{1-w^{2}}{i w}\frac{z}{1+z^{2}} \bigr)}{1+z^{2}}\log\frac{z}{i }\,dz +\frac{2i }{(k-2)!}\int_{i }^{i /w}\frac{\log^{k-2}\bigl(-\frac{1-w^{2}}{i w} \frac{z}{1+z^{2}} \bigr)}{1+z^{2}}\log\frac{z}{i }\,dz.
\end{align*}
This last expression can be treated using the generalized polylogarithms \cite{FTW2016} (also called hyperlogarithms in \cite{Panzer2015}) when $w=\exp(2\pi m i/N)$ after noticing that
\begin{align*}
&{}\log^k\bigg(\frac{1-w^{2}}{w}\bigg)\in \CMZV_{k}^{{\rm lcm}(2,N)}.
\end{align*}
The key step in the process of computing the generalized polylogarithms is to convert various complicated iterated integrals to those that can be seen easily to lie inside $\CMZV$, which is called the fibration in \cite{SunZhou2024sum3k4k} since one of the major programs in Panzer's package \cite{Panzer2014c,Panzer2015} is named \texttt{fibrationBasis}.

Combining with Au's algorithm \cite{Au2020}, Campbell et al. further proved the following in \cite{CampbellGZ2024}.
\begin{thm}
For positive integer $k\ge 2$, set
\begin{align*}
\calS_{k}(z):=\sum_{n=0}^{\infty} \frac{z^{n}}{ (2n+1)^k \binom{2n}{n}}.
\end{align*}
Then the following relations hold true:
\begin{alignat*}{6}
\calS_{k}\biggl( -\frac{9}{4} \biggr)\in &  \CMZV_k^6, &
\quad \calS_{k}(-4)\in & \CMZV_k^8,&
\quad \calS_{k}\biggl(-\frac12\biggr)\in & \sqrt{2}\CMZV_k^8, \\
\calS_{k}(-1)\in & \CMZV_k^{10}, &
\quad \calS_k\biggl( -\frac{16}{5} \biggr)\in & \sqrt{5}\CMZV_k^{10},&
\quad \calS_k\biggl( -\frac{4}{3} \biggr)\in & \sqrt{3}\CMZV_k^{12}.
\end{alignat*}
\end{thm}

By applying similar ideas, Y. Zhou \cite{Zhou2023} showed that for all $k\in\N$
\begin{alignat*}{4}
\sum_{n_1\ge n_2\ge 0}  \frac{ \binom{2n_1}{n_1} }{ 8^{n_1} n_2^k }\in & \sqrt{2}\CMZV_k^8,\qquad &
\sum_{2 n_1\ge n_2\ge 0}  \frac{ \binom{2n_1}{n_1} }{ 8^{n_1} n_2^k }\in & \sqrt{2}\CMZV_k^8,\\
\sum_{n_1\ge n_2\ge 0}  \frac{ \binom{2n_1}{n_1} }{ 16^{n_1} n_2^k }\in & \sqrt{3}\CMZV_k^{12},\qquad &
\sum_{2 n_1\ge n_2\ge 0}  \frac{ \binom{2n_1}{n_1} }{ 16^{n_1} n_2^k }\in & \sqrt{3}\CMZV_k^{12} .
\end{alignat*}
We remark that the method of fibration can be applied to some other Ap\'ery-like sums involving binomial coefficients
$\binom{3n}{n}$, $\binom{4n}{2n}$ and $\binom{6n}{3n}$. We refer the interested readers to \cite{Au2022,CampbellGZ2024,SunZhou2024sum3k4k,Zhou2023} for more details.

\section{Ap\'ery-like sums via contour integrals}
Let $H_n^{(p)}$ be the generalized $n$-th harmonic number of order $p$, defined by
\begin{equation*}
H_n^{(p)} := \sum_{k=1}^n \frac{1}{k^p}, \quad H_n \equiv H_n^{(1)}, \quad\text{and}\quad H_0^{(p)} := 0.
\end{equation*}
In \cite{FS1998}, Flajolet and Salvy introduced the method of contour integration and residue computation to study Dirichlet series involving harmonic numbers of the form:
\begin{align*}
S_{\bfp,q} := \sum_{n = 1}^\infty \frac{H_n^{(p_1)}H_n^{(p_2)} \cdots H_n^{(p_r)} }{n^q},
\end{align*}
which are called \emph{generalized (non-alternating) Euler sums}. When $r > 1$, these are referred to as \emph{nonlinear Euler sums}. Here, $\mathbf{p} := (p_1, p_2, \dots, p_r) \in \mathbb{N}^r$ satisfies $p_1 \le p_2 \le \dots \le p_r$, and $q \ge 2$. The quantity $w := p_1 + \dots + p_r + q$ is called the weight, and $r$ is the degree (or order) of the sum. By the well-known harmonic product (also termed the \emph{quasi-shuffle product} \cite{Hoffman1997}), any generalized Euler sum $S_{\bfp,q}$ can be expressed as a $\Q$-linear combination of MZVs. An explicit formula rewriting generalized Euler sums in terms of MZVs can be found in \cite{XW2021}.

Flajolet and Salvy \cite{FS1998} defined a kernel function $\xi(s)$ by two criteria: (i) it is meromorphic on the entire complex plane, and (ii) it satisfies $\xi(s) = o(s)$ along an infinite sequence of circles $|s| = \rho_k$ with $\rho_k \to \infty$. By leveraging these properties of the kernel function $\xi(s)$, they established the following residue type result.

\begin{lem} \label{lem-residue} \emph{(\cite{FS1998})}
Let $\xi(s)$ be a kernel function, and let $r(s)$ be a rational function that is $O(s^{-2})$ at infinity. Then
\begin{align}\label{3.1}
\sum_{\alpha \in O} \operatorname{Res}\big(r(s)\xi(s), \alpha\big) + \sum_{\beta \in S} \operatorname{Res}\big(r(s)\xi(s), \beta\big) = 0,
\end{align}
where $S$ is the set of poles of $r(s)$, and $O$ is the set of poles of $\xi(s)$ that are not poles of $r(s)$. Here, $\operatorname{Res}(f(s), s_0)$ denotes the residue of $f(s)$ at $s = s_0$.
\end{lem}

Flajolet and Salvy considered contour integrals involving the classical digamma function, using residue computations to establish explicit reductions of generalized Euler sums to Euler sums of lower degree. In particular, by adopting the kernel functions
\begin{equation*}
\xi_r(s):=\pi \cot(\pi s)\frac{\psi^{(p_1-1)}(-s)\cdots\psi^{(p_r-1)}(-s)}{(p_1-1)!\cdots(p_r-1)!}
\end{equation*}
and the base rational function $r(s) = s^{-q}$, they proved the celebrated theorem stating that a nonlinear Euler sum $S_{p_1,\dots,p_r,q}$ reduces to a combination of sums of lower orders whenever the total weight $p_1 + \dots + p_r + q$ and the degree $r$ have the same parity.

\subsection{Main Methods}
By extending Flajolet and Salvy's idea it is not too difficult to discover the following fact that any
polynomial form in
\begin{equation*}
\pi\cot\pi s,\quad \frac{\pi}{\sin\pi s},\quad \psi^{(j)}(\pm s),\quad
\frac{\Gamma^2(s+1)}{\Gamma(2s+1)}4^s
\end{equation*}
is itself a kernel function with poles at a subset of the integers. Because by the Legendre duplication formula
\[
\Gamma(s)\Gamma\smbb{s+\frac{1}{2}}=\sqrt{\pi}\cdot 2^{1-2s}\Gamma(2s)
\]
and the asymptotic expansion for the ratio of two gamma functions
\[
\frac{\Gamma(s+a)}{\Gamma(s+b)}
    =s^{a-b}\smbb{1+O\smbb{\frac{1}{s}}}\,,
    \quad\text{for } |\arg(s)|\leq\pi-\varepsilon\,,\ \varepsilon>0\,,\ |s|\to\infty
\]
(see \cite[Sections 2.3 and 2.11]{Luke69.1}), we have
\begin{align}\label{eq-inf-Gamma}
\frac{\Gamma^2(s+1)}{\Gamma(2s+1)}4^s=\sqrt{\pi s}\smbb{1+O\smbb{\frac{1}{s}}}=o(s)\,,\quad |s|\to\infty\,.
\end{align}
The purpose of this section is to investigate the efficacy of these kernel functions in establishing summation formulas and evaluating Ap\'ery-like sums. To this end, we require the following lemmas. First, recall that (see \cite[Section 3.3]{Com74} and \cite[Section 2.8]{Riordan58}) the \emph{exponential complete Bell polynomials} $Y_n$ are defined by
\begin{equation}\label{cBell.gf}
\exp\smbb{\sum_{k=1}^{\infty}x_k\frac{t^k}{k!}}
    =\sum_{n=0}^{\infty}Y_n(x_1,x_2,\ldots,x_n)\frac{t^n}{n!}\,,
\end{equation}
and they satisfy the recurrence
\begin{equation}\label{cBell.rec}
Y_n(x_1,x_2,\ldots,x_n)=\sum_{j=0}^{n-1}\binom{n-1}{j}x_{n-j}Y_j(x_1,x_2,\ldots,x_j)
    \,,\quad n\geq1.
\end{equation}

\begin{lem} \emph{(\cite{WangXu2021})} \label{Lem.CD}
For any complex number $s$ satisfying $|s|<1$, the following identities hold:
\[
\Gamma(s+1)\uue^{\gamma s}=\sum_{n=0}^\infty C_n\frac{s^n}{n!}\,,\quad
\{\Gamma(s+1)\uue^{\gamma s}\}^{-1}=\sum_{n=0}^\infty D_n\frac{s^n}{n!}\,,
\]
where $\gamma:=\lim_{n\to\infty}(H_n-\log n)$ is the Euler-Mascheroni constant, and
\begin{align*}
&C_n:=Y_n(0,1!\ze(2),-2!\ze(3),\ldots,(-1)^{n}(n-1)!\ze(n))\,,\\
&D_n:=Y_n(0,-1!\ze(2),2!\ze(3),\ldots,(-1)^{n-1}(n-1)!\ze(n))\,.
\end{align*}
\end{lem}

According to Lemma \ref{Lem.CD}, it is clear that $C_k$ and $D_k$ are rational linear combinations of products of zeta values. In particular, by \eqref{cBell.rec}, we have
$(C_k)_{k\in\mathbb{N}_0}=(1,0,\ze(2),-2\ze(3),\frac{27}{2}\ze(4),\ldots)$ and $(D_k)_{k\in\mathbb{N}_0}=(1,0,-\ze(2),2\ze(3),\frac{3}{2}\ze(4),\ldots)$.
Furthermore, by the relations
\begin{align*}
&\frac{\Gamma(s-n)}{\Gamma(s+1)}
    =\prod_{k=0}^n\frac{1}{s-k}
    =\frac{(-1)^n}{n!}\frac{1}{s}\sum_{m=0}^\infty\ze_n^\star(\{1\}_m)s^m\,,\\
&\frac{\Gamma(s+1)}{\Gamma(s-n)}
    =\prod_{k=0}^n(s-k)
    =(-1)^nn!s\sum_{m=0}^\infty(-1)^m\ze_n(\{1\}_m)s^m\,,
\end{align*}
for $n\geq 0$, the next lemma can be established.

\begin{lem}\label{Lem.AB}
For any nonnegative integer $n$, as $s\to-n$ we have
\begin{align*}
&\Gamma(s)\uue^{\gamma(z-1)}
    =\frac{(-1)^n}{n!}\uue^{-\gamma(n+1)}\sum_{k=0}^\infty A_k(n)(s+n)^{k-1}\,,\\
&\frac{1}{\Gamma(s)\uue^{\gamma(s-1)}}
    =(-1)^nn!\uue^{\gamma(n+1)}\sum_{k=0}^\infty B_k(n)(s+n)^{k+1}\,,
\end{align*}
where
\[
A_k(n):=\sum_{\substack{k_1+k_2=k\\k_1,k_2\geq 0}}
    H_n^\star(\{1\}_{k_1})\frac{C_{k_2}}{k_2!}\,,\quad
B_k(n):=\sum_{\substack{k_1+k_2=k\\k_1,k_2\geq 0}}
    (-1)^{k_1} H_n(\{1\}_{k_1})\frac{D_{k_2}}{k_2!}\,.
\]
\end{lem}
Here, for $\bfk=(k_1,k_2,\ldots,k_r)\in\N^r$ the multiple harmonic sum $H_n(\bfk)$ is defined by \eqref{equ:MHS} and the \emph{multiple harmonic star sum} $H_n^\star(\bfk)$ is defined by
\begin{align}\label{defn-mhsss}
H_n^\star(\bfk):=
    \sum_{n\geq n_1\geq n_2\geq\cdots\geq n_r\geq 1}
    \frac{1}{n_1^{k_1}n_2^{k_2}\cdots n_r^{k_r}}\,.
\end{align}
Here, if $n<r$ then $H_n(\bfk):=0$ and $H_n(\emptyset)=H^\star_n(\emptyset):=1$. When taking the limit $n\rightarrow \infty$ in \eqref{equ:MHS} and \eqref{defn-mhsss} with $k_1>1$, we recover the multiple zeta values and the multiple zeta star values, respectively.

Now, we use Lemmas \ref{lem-residue} and \ref{Lem.AB} to establish some explicit relations of Ap\'ery-like sums.
For more details we refer the readers to \cite{WangXu2021}.

\subsection{Evaluations of some Ap\'ery-like sums via contour integrals}
Let
\begin{equation}
G(\vec{k}):=G(k_1,k_2,k_3,k_4)
    =\frac{2^{k_3+k_4}\log^{k_4}(2)C_{k_1}C_{k_2}D_{k_3}}{k_1!k_2!k_3!k_4!}\,.
\end{equation}
Then we have a new closed form expression.

\begin{thm}\label{Th.Sq}
For any integer $q\geq 2$,  we have
\[
\sum_{n=1}^\infty\frac{\binom{2n}{n}}{4^n n^{q-1}}
    =(-1)^q\sum_{k_1+k_2+k_3+k_4=q-1,\atop k_1,\ldots,k_4\geq 0}G(k_1,k_2,k_3,k_4)\,.
\]
\end{thm}
\begin{proof}
Consider the function
\[
F_0(s):=\frac{\Gamma^2(s+1)}{\Gamma(2s+1)}\frac{4^s}{s^q}\,.
\]
Applying \eqref{eq-inf-Gamma} yields
\[
F_0(s)=\frac{\sqrt{\pi}}{s^{q-1/2}}\smbb{1+O\smbb{\frac{1}{s}}}\,,\quad |s|\to\infty\,,
\]
which implies that the integral $\oint_{(\infty)}F_0(s)\uud s=0$, where $\oint_{(\infty)}$ denotes integration along large circles, that is, the limit of integrals $\oint_{|s|=\rho_k}$. Hence, we can use Lemma \ref{lem-residue} to obtain an identity of infinite series.
On the other hand, the function $F_0(s)$ has poles only at all non-positive integers. For a negative integer $-n$, by Lemma \ref{Lem.AB}, as $s\to -n$ we have
\[
F_0(s)=-\frac{4^s}{s^q}\frac{(2n-1)!}{(n-1)!^2}
    \sum_{k_1,k_2,k_3\geq0}2^{k_3+1}A_{k_1}(n-1)A_{k_2}(n-1)B_{k_3}(2n-1)
    (s+n)^{k_1+k_2+k_3-1}\,.
\]
Hence, the residue is
\[
\operatorname{Res}(F_0(s),-n)=(-1)^{q+1}\frac{\binom{2n}{n}}{4^nn^{q-1}}\,,
    \quad\text{for }n=1,2,\ldots\,.
\]
Similarly, as $s\to 0$ we have
\[
F_0(s)=\frac{4^s}{s^q}\sum_{k_1,k_2,k_3\geq0}
    \frac{2^{k_3}C_{k_1}C_{k_2}D_{k_3}}{k_1!k_2!k_3!}s^{k_1+k_2+k_3}\,,
\]
and the residue of the pole of order $q$ at 0 is
\[
\operatorname{Res}(F_0(s),0)=\sum_{|k|_4=q-1}G(k_1,k_2,k_3,k_4)\,.
\]
Using Lemma \ref{lem-residue} and adding these two contributions together yields the statement of the theorem.
\end{proof}

\begin{thm}\label{ACB1}
For any positive integer $q>1$,
\begin{align}
\sum_{n=1}^\infty \frac{n}{(n-1/2)^q}\cdot \frac{\binom{2n}{n}}{4^n}=\frac{2}{\pi} \sum_{k_1+\cdots+k_4+2k_5=q-2,\atop k_1,\ldots,k_5\geq 0} G(k_1,k_2,k_3,k_4)\widetilde{t}(2k_5+2),
\end{align}
where ${\widetilde{t}(k)}:=2^k \sum\limits_{n=1}^\infty \frac 1{(2n-1)^k}=(2^k-1)\ze(k)$ for $k>1$.
\end{thm}
\begin{proof} Similarly to Theorem \ref{Th.Sq}, we consider the contour integral
\begin{align*}
\oint_{(\infty)}F_1(s)ds:=\oint_{(\infty)} \pi \tan(\pi s) F_0(s)ds=0.
\end{align*}
The function $F_1(s)$ has poles only at $0$ and $n-1/2\ (n=1,2,3,\ldots)$. By straightforward calculations, we obtain
\begin{align*}
\operatorname{Res}(F_1(s),n-1/2)=-\pi \frac{n}{(n-1/2)^q}\frac{\binom{2n}{n}}{4^n}
\end{align*}
and
\begin{align*}
\operatorname{Res}(F_1(s),0)=2\sum_{k_1+\cdots+k_4+2k_5=q-2,\atop k_1,\ldots,k_5\geq 0} \frac{C_{k_1}C_{k_2}D_{k_3}2^{k_3+k_4}(\log(2))^{k_4}{\widetilde{t}(2k_5+2)}}{k_1!k_2!k_3!k_4!}.
\end{align*}
Hence, we complete the proof of Theorem \ref{ACB1}.
\end{proof}

It is possible that some other identities of Ap\'ery-like sums can be proved using the
techniques of the present section. For example, one may also consider the function
$F_2(s):=(\psi(-s)+\gamma)F_0(s)$ which has poles only at all integers. By residue computations, we obtain
\begin{align}
\sum_{n=1}^\infty\frac{4^n}{n^q\binom{2n}{n}}
    =(-1)^q\sum_{n=1}^\infty \frac{H_{n-1}\binom{2n}{n}}{4^nn^{q-1}}+\sum_{k_1+k_2+\cdots+k_5=q-2,\atop k_1,\ldots,k_5\geq 0}G(k_1,k_2,k_3,k_4)\ze(k_5+2)\,.
\end{align}

\end{document}